\documentclass[journal]{IEEEtran}
\AtBeginDvi{} 

\usepackage{color}
\usepackage[pdftex]{graphicx}
\usepackage{framed}
\usepackage{bm}
\usepackage{comment}
\usepackage[verbose]{cite}
\usepackage{algorithm}
\usepackage{algorithmic}
\usepackage{slashbox}
\newtheorem{theorem}{Theorem}

\newtheorem{remark}{Remark}
\newtheorem{lemma}{Lemma}
\newtheorem{corollary}{Corollary}

\newtheorem{proposition}{Proposition}

\newtheorem{problem}{Problem}

\newtheorem{proof}{Proof}

\usepackage{latexsym}
\def\qed{\hfill $\Box$} 

\usepackage{ascmac}

\ifCLASSINFOpdf
\else
\fi
\usepackage{amsmath}
\begin{document}
%
\title{Riemannian optimal identification method for linear systems with symmetric positive-definite matrix}
%
%
%

\author{Kazuhiro~Sato, Hiroyuki~Sato, and Tobias~Damm
\thanks{K. Sato is with the Division of Information and Communication Engineering,
 Kitami Institute of Technology,
 Hokkaido 090-8507, Japan,
email: ksato@mail.kitami-it.ac.jp}
\thanks{H. Sato is with the Department of Applied Mathematics, Graduate School of Informatics, Kyoto University, Kyoto 606-8501, Japan,
email: hsato@amp.i.kyoto-u.ac.jp}
\thanks{T. Damm is with the Department of Mathematics, University of Kaiserslautern, Kaiserslautern, Germany,
email: damm@mathematik.uni-kl.de}
}

\maketitle
\thispagestyle{empty}
\pagestyle{empty}

\begin{abstract}
This study develops identification methods for linear continuous-time symmetric systems,
such as electrical network systems, multi-agent network systems, and temperature dynamics in buildings.
To this end, we formulate three system identification problems for the corresponding discrete-time systems.
The first is a least-squares problem in which we wish to minimize the sum of squared errors between the true and model outputs on the product manifold of
the manifold of symmetric positive-definite matrices and two Euclidean spaces.
In the second problem, to reduce the search dimensions, the product manifold is replaced with the quotient set under a specified group action by the orthogonal group.
In the third problem, the manifold of symmetric positive-definite matrices in the first problem is replaced by the manifold of matrices with only positive diagonal elements.
In particular, we examine the quotient geometry in the second problem.
We propose Riemannian conjugate gradient methods for the three problems,
and select initial points using a popular subspace method.
The effectiveness of our proposed methods is demonstrated through numerical simulations and comparisons with the Gauss--Newton method, which is one of the most popular approach for solving least-squares problems.
\end{abstract}


%
\IEEEpeerreviewmaketitle

\section{Introduction} \label{sec1}
%
%
%
%
\IEEEPARstart{M}{any} important systems involved in electrical networks \cite{dorfler2018electrical, vandenberghe1997optimal, willems1976realization}, multi-agent networks \cite{mesbahi2010graph, olfati2004consensus}, and temperature dynamics in buildings \cite{hatanaka2017physics, lidstrom2016optimal} can be modeled as  
\begin{align}
\begin{cases}
\dot{\hat{x}}(t) = F\hat{x}(t) + G\hat{u}(t), \\
\hat{y}(t) = H\hat{x}(t),
\end{cases} \label{1}
\end{align}
where $\hat{x}(t)\in {\bf R}^n$, $\hat{u}(t)\in {\bf R}^m$, and $\hat{y}(t)\in {\bf R}^p$ are the state, input, and output of the system, respectively,
and $F\in {\rm Sym}(n)$, $G\in {\bf R}^{n\times m}$, and $H\in {\bf R}^{p\times n}$ are constant matrices.
Because the matrix $F$ is symmetric, we call \eqref{1} a linear continuous-time symmetric system.
In controlling a system described by \eqref{1}, it is important to have an
 accurate identification of $(F,G,H)$.

However, no identification methods for system \eqref{1} have yet been developed.
More specifically, although many identification techniques have been developed for discrete-time systems, such as prediction error methods \cite{hanzon1993riemannian, ljung1999system, mckelvey2004data, peeters1994system, sato2017riemannian, wills2008gradient}
and subspace identification methods \cite{katayama2006subspace, larimore1990canonical, qin2006overview, van1994n4sid, van1995unifying, verhaegen1992subspace, verhaegen2016n2sid}
for discrete-time systems, as well as for continuous-time systems \cite{campi2008iterative, garnier2008contsid, johansson1999stochastic, maruta2013projection, ohsumi2002subspace, van1996continuous},
it is difficult to identify a symmetric matrix $F$ from the $K+1$ input/output data $\{ (u_0,y_0), (u_1,y_1),\ldots, (u_{K},y_{K})\}$ over the sampling interval $h$.
Here, $y_0 ,y_{1},\ldots, y_{K}$ are noisy data observed from the true system, which is different from \eqref{1}.
This is because no system identification method has been derived for the corresponding discrete-time system 
\begin{align}
\begin{cases}
\hat{x}_{k+1} = A \hat{x}_k + B \hat{u}_k, \\
\hat{y}_k = C\hat{x}_k,
\end{cases} \label{2}
\end{align}
where $\hat{x}_k:=\hat{x}(kh)$, $\hat{u}_k:= \hat{u}(kh)$, $\hat{y}_k:=\hat{y}(kh)$, and
\begin{align}
A &:= \exp( Fh) \in {\rm Sym}_+(n), \label{3} \\
B &:= \left( \int_0^h \exp(Ft) dt \right) G, \label{4} \\
C &:= H. \label{5}
\end{align}
That is, the existing methods in \cite{hanzon1993riemannian, ljung1999system, mckelvey2004data, peeters1994system, sato2017riemannian, wills2008gradient, katayama2006subspace, larimore1990canonical, qin2006overview, van1994n4sid, van1995unifying, verhaegen1992subspace, verhaegen2016n2sid}  for identifying the triplet $(A,B,C)$ do not provide a symmetric positive-definite matrix $A$.

For this reason, we present novel prediction error methods for identifying
\begin{align*}
\Theta:=(A,B,C)\in M:= {\rm Sym}_+(n)\times {\bf R}^{n\times m}\times {\bf R}^{p\times n}
\end{align*}
using the input/output data $\{ (u_0,y_0), (u_1,y_1),\ldots, (u_{K},y_{K})\}$
under the assumption that the matrix $A$ is stable.
That is, we identify the matrix $A$ to be symmetric positive definite.
If this is achieved, we can also obtain the matrices $F$, $G$, and $H$ by
\begin{align}
F &= \log A/h, \label{F}\\
G &= \left( \int_0^h \exp(Ft) dt \right)^{-1} B, \label{G}\\
H &= C. \label{H}
\end{align} 
In particular, the matrix $F$ is symmetric and is uniquely determined, because the map 
$\exp: {\rm Sym}(n)\rightarrow {\rm Sym}_+(n)$ is bijective \cite{lang2012fundamentals}.

To develop the prediction error methods, we formalize three different problems by introducing the Riemannian metric
\begin{align}
&\quad \langle (\xi_1,\eta_1,\zeta_1),  (\xi_2,\eta_2,\zeta_2) \rangle_{\Theta}  \nonumber\\
&:= {\rm tr}(A^{-1}\xi_1A^{-1}\xi_2) + {\rm tr}(\eta_1^{\top}\eta_2) + {\rm tr}(\zeta_1^{\top} \zeta_2) \label{metric_M}
\end{align}
for $(\xi_1,\eta_1,\zeta_1),  (\xi_2,\eta_2,\zeta_2)\in T_{\Theta} M$, where the metric has also been used for a model reduction problem \cite{sato2018TAC}.
The first problem is the least-squares problem of minimizing the sum of squared errors on the Riemannian manifold $M$.
In the second problem, to reduce the search dimension of the first problem,
the manifold $M$ is replaced by a quotient set.
In the third problem, we replace the ${\rm Sym}_+(n)$ component of $M$ with ${\rm Diag}_+(n)$.

The contributions of this paper are summarized as follows.\\
1) In Section \ref{SecIIB}, we show that 
the quotient set $N/O(n)$ in the second problem is indeed a manifold, where 
\begin{align}
N:= M \cap S_{\rm con}\cap S_{\rm ob}. \label{N_def}
\end{align}
Here,
$S_{\rm con} :=\{ (A,B,C)\in {\bf R}^{n\times n}\times {\bf R}^{n\times m}
\times {\bf R}^{p\times n}\,| \,  (A,B,C)\,\, {\rm is\,\, controllable}\}$ and
$S_{\rm ob} :=\{ (A,B,C)\in {\bf R}^{n\times n}\times {\bf R}^{n\times m}
\times {\bf R}^{p\times n}\,| \,  (A,B,C)\,\, {\rm is\,\, observable}\}$,
where we say that $(A,B,C)$ is controllable (resp. observable) if the corresponding discrete-time system
described by \eqref{2} is controllable (resp. observable).
Moreover, in Section \ref{SecIIB}, we prove that Riemannian metric \eqref{metric_M}
 on $M$ induces a Riemannian metric into $N/O(n)$
by using a general theorem, as shown in Appendix \ref{apeA}.
That is, the quotient set $N/O(n)$ is shown to be a Riemannian manifold.

\noindent
2) In Section \ref{Sec4}, we propose Riemannian conjugate gradient (CG) methods for solving the aforementioned three problems.
In developing the CG method for the first problem, we derive the Riemannian gradient of the objective function in terms of Riemannian metric \eqref{metric_M},
and use the concept of parallel transport.
For the modified second problem on the quotient manifold $N/O(n)$, the parallel transport in the first problem is replaced by the projection onto the horizontal space that is a subspace of a tangent space of the manifold $N$,
although the Riemannian gradient is the same.
In Section \ref{Sec3C2}, it is shown that the projection is obtained using the skew-symmetric solution to a linear matrix equation.
In Appendix \ref{Proof_Thm2}, we prove that there exists a unique skew-symmetric solution to the equation under a mild assumption.
Moreover, for the third problem, we derive another Riemannian gradient different from that in the first and second problems.
Furthermore, in Section \ref{Sec4D}, we propose a technique for choosing initial points in the proposed algorithms for solving the three problems
 based on a subspace method such as N4SID \cite{van1994n4sid}, MOESP \cite{verhaegen1992subspace}, CVA \cite{larimore1990canonical}, ORT \cite{katayama2006subspace}, or N2SID \cite{verhaegen2016n2sid}.

\noindent
3) We demonstrate the effectiveness of our proposed methods for single-input single-output (SISO) and multi-input multi-output (MIMO) cases:
\begin{itemize}
\item  Our proposed methods for solving the aforementioned three problems can produce $A\in {\rm Sym}_+(n)$, unlike the Gauss--Newton (GN) method, which has been widely used for solving least-squares problems.
In other words, we illustrate that the usual GN method as explained in Section \ref{Sec5} is not adequate for identifying system \eqref{1}.
\item Our proposed methods significantly improve the results produced by a modified MOESP method in terms of various indices.
\item In MIMO cases, the rate of instability in the estimated matrix $A_{\rm est}$ produced by our method when solving the third problem is much higher than that for solving the first and second problems.
In other words, the proposed methods for solving the first and second problems have a high degree of stability.
\item  A hybrid approach combining the CG methods for solving the first and second problems may be more efficient than the individual CG methods.
\end{itemize}

The remainder of this paper is organized as follows.
In Section \ref{Sec2}, we formulate the aforementioned three problems mathematically.
In particular, in Section \ref{SecIIB}, we show that the quotient set $N/O(n)$ is a manifold.
Moreover, we prove that Riemannian metric \eqref{metric_M} on $M$ induces a Riemannian metric on $N/O(n)$.
In Section \ref{Sec3}, we discuss Riemannian geometries of our problems.
In Section \ref{Sec4}, we propose optimization algorithms for solving the three problems.
In addition, we propose a technique for choosing an initial point in the algorithms.
In Section \ref{Sec5}, we summarize the GN method.
In Section \ref{Sec6}, we demonstrate the effectiveness of our proposed methods.
Finally, the conclusions of this study are presented in Section \ref{Sec7}.

{\it Notation:} The sets of real and complex numbers are denoted by ${\bf R}$ and ${\bf C}$, respectively.
The symbols ${\rm Sym}(n)$ and ${\rm Skew}(n)$ denote the vector spaces of symmetric matrices and skew-symmetric matrices in ${\bf R}^{n\times n}$, respectively.
The symbol ${\rm Diag}(n)$ is the vector space of diagonal matrices in ${\bf R}^{n\times n}$.
The manifold of symmetric positive definite matrices in ${\rm Sym}(n)$ is denoted by ${\rm Sym}_+(n)$.
The manifold of matrices with positive diagonal elements in ${\rm Diag}(n)$ is denoted by ${\rm Diag}_+(n)$.
The symbol $O(n)$ denotes the orthogonal group in ${\bf R}^{n\times n}$.
The tangent space at $p$ on a manifold $\mathcal{M}$ is denoted by $T_p \mathcal{M}$.
The identity matrix of size $n$ is denoted by $I_n$.
Given vectors $v=(v_i), w=(w_i)\in {\bf R}^n$, $(v,w)$ denotes the Euclidean inner product, i.e.,
$(v,w) = \sum_{i=1}^n v_i w_i$,
and
$||v||_2$ denotes the Euclidean norm, i.e.,
$||v||_2:= \sqrt{(v,v)} = \sqrt{v_1^2+v_2^2+\cdots +v_n^2}$.
Given a matrix $A\in {\bf R}^{n\times n}$, $||A||_F$ denotes the Frobenius norm, i.e.,
$||A||_F:= \sqrt{ {\rm tr} (A^{\top} A)}$,
where the superscript $\top$ denotes the transpose and ${\rm tr} (A)$ denotes the trace of $A$, i.e., the sum of the diagonal elements of $A$.
The symbol
$\lambda(A)$ denotes the set of eigenvalues of $A$,
and ${\rm sym}(A)$ and ${\rm sk}(A)$ denote the symmetric and skew-symmetric parts of $A$, respectively, i.e., ${\rm sym}(A)=\frac{A+A^{\top}}{2}$ and ${\rm sk}(A)=\frac{A-A^{\top}}{2}$.
Given a smooth function $f$ between finite dimensional manifolds $\mathcal{M}$ and $\mathcal{N}$,
the differential of $f$ at $x$ is denoted by ${\rm D}f(x)$.

\section{Problem settings} \label{Sec2}

This section presents the formulation of the three problems.

\subsection{Problem 1}

As described earlier, the aim of this study is to develop a novel prediction error method for identifying $\Theta\in M$ using the input/output data.
To this end, we consider the following problem.
\begin{framed}
\begin{problem}
Suppose that the input/output data $\{ (u_0,y_0), (u_1,y_1),\ldots, (u_K,y_K)\}$ and state dimension $n$ are given.
Then, find the minimizer of
\begin{align*}
&{\rm minimize}\quad f_1(\Theta):= ||e(\Theta)||^2_2 \\
&{\rm subject\, to}\quad \Theta\in M.
\end{align*}
\end{problem}
\vspace{-.4\baselineskip}
\end{framed}

\noindent
Here, 
\begin{align}
 e(\Theta) :=\begin{pmatrix}
y_1-\hat{y}_1(\Theta)  \\
y_2-\hat{y}_2(\Theta) \\
\vdots \\
y_K-\hat{y}_K(\Theta) 
\end{pmatrix}\in {\bf R}^{pK}, \label{pe_vector}
\end{align}
and $\hat{y}_k(\Theta)$ is $\hat{y}_k$ obtained by substituting the input data $u_k$ into $\hat{u}_k$ of \eqref{2}.
The initial state $\hat{x}_0\in {\bf R}^n$ in \eqref{2} is arbitrary.
Note that $\hat{y}_k(\Theta)$ is different from the output data $y_k$, which is obtained by observing the output of the true system.
That is,
\eqref{2} is a mathematical model but is not the true system.

In this paper, as mentioned in Section \ref{sec1}, we endow Riemannian metric \eqref{metric_M} into $M$.
Thus, Problem 1 is a Riemannian optimization problem.

\subsection{Problem 2} \label{SecIIB}

It is possible to reduce the dimension of the problem of minimizing $||e(\Theta)||_2^2$ under the assumption that the initial state $\hat{x}_0$ is equal to zero.
This is because $\Theta$ and 
\begin{align*}
U\circ \Theta:=(U^{\top}AU,U^{\top}B,CU)
\end{align*}
realize input/output equivalent systems for any $U\in O(n)$,
where $\circ$ denotes a group action of $O(n)$ on $M$.
That is, they attain the same value of the prediction error, i.e., $||e(\Theta)||_2 = ||e(U\circ \Theta)||_2$.
Moreover, if $\Theta\in M$, we have $U\circ \Theta\in M$ for any $U\in O(n)$.
This leads to the idea of equating $\Theta$ with $U\circ \Theta$ to reduce the dimension of the problem of minimizing $||e(\Theta)||_2^2$.

To this end, we endow $M$ with an equivalence relation $\sim$, where $\Theta_1 \sim \Theta_2$ if and only if
there exists some $U\in O(n)$ such that $\Theta_2 = U\circ \Theta_1$.
Defining the equivalence class $[\Theta]$ by $[\Theta] := \{ \Theta_1\in M | \Theta_1\sim \Theta\}$,
we can equate $\Theta$ with any $\Theta_1$ that is equivalent to $\Theta$.
Thus, instead of Problem 1, we can consider a minimization problem on the quotient set
$M/O(n):= \{ [\Theta] | \Theta \in M\}$.

However, it is an open problem whether the quotient set $M/O(n)$ is a manifold,
although this set is a Hausdorff space from Proposition 21.4 in \cite{lee2013intro}.
In fact, although there are topological studies on control systems \cite{afsari2017bundle, brockett1976some, delchamps1985global, hanzon1998overlapping, hazewinkel1977moduli},
there is no existing work on the quotient set $M/O(n)$.
Thus, it is difficult to guarantee that $\pi_M^{-1}([\Theta])$ is a submanifold of $M$ for all $\Theta\in M$,
because we cannot use well-known general results such as Proposition 3.4.4 in \cite{absil2008optimization}.
Here, the map $\pi_M:M\rightarrow M/O(n)$ denotes the canonical projection, i.e., $\pi_M (\Theta)=[\Theta]$ for any $\Theta\in M$.
If $\pi_M^{-1}([\Theta])$ is not a manifold for some $\Theta\in M$, then $T_{\Theta}\pi_M^{-1}([\Theta])$ cannot be defined.
That is, in this case, we cannot consider the vertical space in $T_{\Theta} M$.
As a result, it may be impossible to define the horizontal space that is the orthogonal complement of the vertical space with respect to metric \eqref{metric_M}.
This makes it difficult to develop an optimization method for solving the problem.

To resolve this issue, we consider the set $N$ defined by \eqref{N_def}
instead of $M$.
The set $N$ is an open submanifold of ${\bf R}^{n\times n}\times {\bf R}^{n\times m} \times {\bf R}^{p\times n}$,
because, in addition to $M$, $S_{\rm con}$ and $S_{\rm ob}$ are open sets in ${\bf R}^{n\times n}\times {\bf R}^{n\times m} \times {\bf R}^{p\times n}$, as shown in Proposition 3.3.12 in \cite{sontag1998mathematical}.
A group action of $O(n)$ on $N$, as in $M$, is given by
\begin{align}
U\circ \Theta:=(U^{\top}AU,U^{\top}B,CU), \label{action}
\end{align}
where $\Theta\in N$.
Then, $U\circ \Theta \in N$ for any $U\in O(n)$. 
By introducing the equivalence class $[\Theta]:= \{ \Theta_1\in N\,|\, \Theta_1\sim \Theta\}$, 
we can define the quotient set $N/O(n):= \{ [\Theta]\,|\, \Theta\in N\}$.

Unlike $M/O(n)$, we can guarantee that $N/O(n)$ is a manifold using the quotient manifold theorem \cite{lee2013intro}, which is explained in Appendix \ref{ape_quotient}.
To see this, we must confirm that action \eqref{action} is free and proper.
Action \eqref{action} is proper because the Lie group $O(n)$ is compact.
For a more detailed explanation, see Corollary 21.6  in \cite{lee2013intro}.
Thus, we show that action \eqref{action} is free.
Suppose that the general linear group $GL(n)$ acts on $N$ as
\begin{align*}
T\diamondsuit \Theta:=(T^{-1}AT,T^{-1}B,CT),\quad T \in GL(n),\,\, \Theta \in N.
\end{align*}
This action is free, as explained in Remark 6.5.10 in \cite{sontag1998mathematical}.
That is,
\begin{align}
\{T\in GL(n)\,|\, T\diamondsuit \Theta = \Theta\} = \{I_n\} \label{15_2}
\end{align}
for any $\Theta\in N$.
Moreover, we have that
\begin{align}
\{ I_n\} \subset \{U\in O(n) | U\circ \Theta = \Theta \} \subset \{T\in GL(n) | T\diamondsuit \Theta = \Theta\} \label{16_2}
\end{align}
for any $\Theta\in N$.
From \eqref{15_2} and \eqref{16_2}, action \eqref{action} is free.

Thus, the following problem is an optimization problem on a manifold.

\begin{framed}
\begin{problem}
Suppose that the input/output data $\{ (u_0,y_0), (u_1,y_1),\ldots, (u_K,y_K)\}$ and state dimension $n$ are given.
Then, find the minimizer of
\begin{align*}
&{\rm minimize}\quad f_2([\Theta]):= ||e(\Theta)||^2_2 \\
&{\rm subject\, to}\quad [\Theta]\in N/O(n).
\end{align*}
\end{problem}
\vspace{-.4\baselineskip}
\end{framed}

That is, we also develop a prediction error method on the quotient manifold $N/O(n)$.
Note that this development is different from that in \cite{sato2017riemannian}, which 
considered a group action of the general linear group $GL(n)$ on a manifold instead of that of $O(n)$.
It is not adequate to use the action in \cite{sato2017riemannian} for our problem, because this action does not, in general, preserve the symmetric positive-definiteness of the matrix $A$.
For this reason, we consider the group action of $O(n)$ on the manifold $N$.

To introduce a Riemannian metric into $N/O(n)$, we define Riemannian metric \eqref{metric_M} on $N$.

Because $N/O(n)$ is a quotient manifold, Proposition 3.4.4 in \cite{absil2008optimization} implies that $\pi^{-1}([\Theta])$ is an embedded submanifold of $N$ for any $\Theta\in N$, where
the map $\pi:N\rightarrow N/O(n)$ denotes the canonical projection, i.e., $\pi (\Theta)=[\Theta]$ for any $\Theta\in N$.
Thus, we can define the vertical space $\mathcal{V}_{\Theta}:=T_{\Theta} \pi^{-1}([\Theta])$ in $T_{\Theta} N$ for any $\Theta\in N$.
Moreover, from Proposition 3.9 in \cite{lee2013intro}, $T_{\Theta} M = T_{\Theta} N$ for any $\Theta \in N$,
because $N$ is an open set in $M$.
Hence, we can consider $\mathcal{V}_{\Theta}$ in $T_{\Theta} M$ for any $\Theta \in N\subset M$.
Additionally, the horizontal space $\mathcal{H}_{\Theta}$  can be defined as the orthogonal complement of the vertical space 
$\mathcal{V}_{\Theta}$ in $T_{\Theta}N$ with respect to metric \eqref{metric_M}.
Furthermore, the horizontal lift $\bar{\xi}_{\Theta}\in \mathcal{H}_{\Theta}$ of $\xi\in T_{[\Theta]}(N/O(n))$ is defined as the unique element of the horizontal space $\mathcal{H}_{\Theta}$ satisfying ${\rm D}\pi(\Theta)[\bar{\xi}_{\Theta}] =\xi$.
Fig.\,\ref{Fig0} presents a diagram of these concepts.

\begin{figure}[t]
\begin{center}
\includegraphics[scale=0.4]{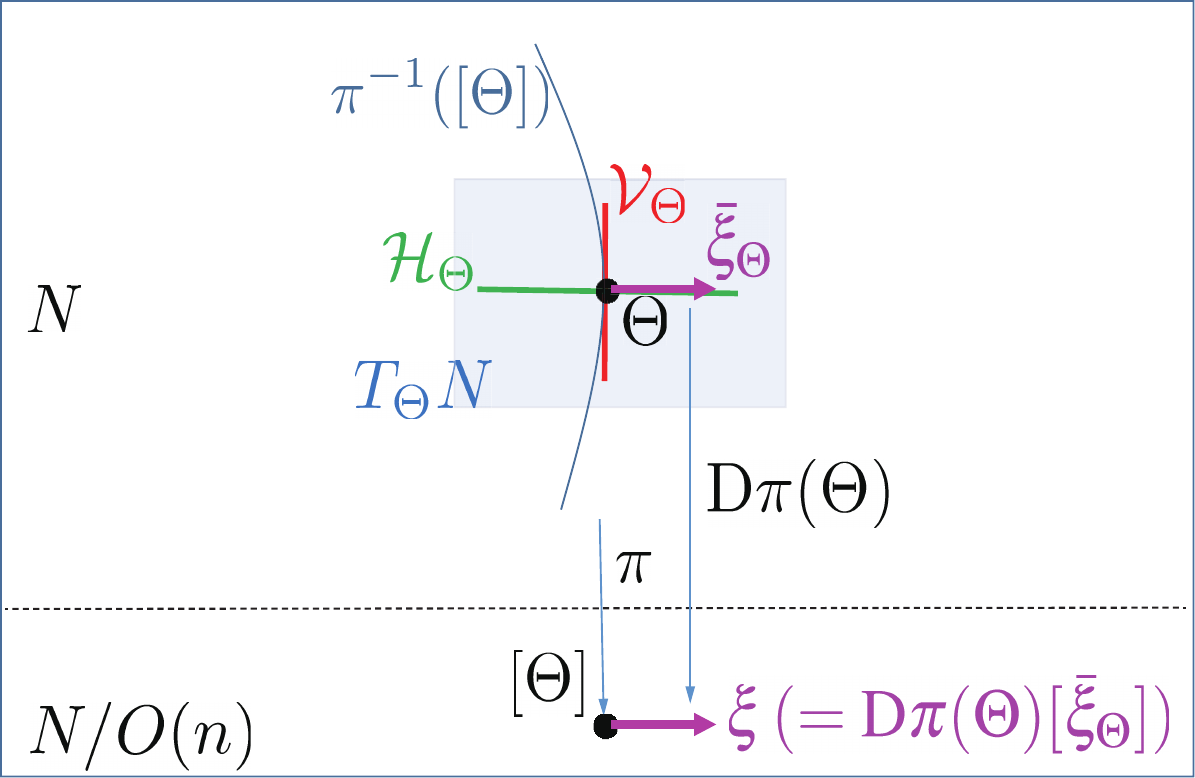}
\end{center}
\vspace{-1em}
\caption{Conceptual diagram of vertical space $\mathcal{V}_{\Theta}$, horizontal space $\mathcal{H}_{\Theta}$, and horizontal lift $\bar{\xi}_{\Theta}$.} \label{Fig0}
\end{figure}

In the following, we show that a Riemannian metric on $N/O(n)$ can be defined by
\begin{align}
\langle \xi, \zeta \rangle_{[\Theta]} := \langle \bar{\xi}_{\Theta}, \bar{\zeta}_{\Theta} \rangle_{\Theta}, \label{metric_quotient}
\end{align}
where $\xi,\zeta \in T_{[\Theta]} (N/O(n))$, $\Theta\in \pi^{-1}([\Theta])$, and $\bar{\xi}_{\Theta}$ and $\bar{\zeta}_{\Theta}$ are the horizontal lifts of $\xi$ and $\zeta$ at $\Theta\in N$, respectively.
Note that $\langle \cdot,\cdot \rangle_{\Theta}$ of the right-hand side of \eqref{metric_quotient} is Riemannian metric \eqref{metric_M}.

To this end, we must prove that 
\begin{align}
\langle \bar{\xi}_{\Theta_1}, \bar{\zeta}_{\Theta_1} \rangle_{\Theta_1} = \langle \bar{\xi}_{\Theta_2}, \bar{\zeta}_{\Theta_2} \rangle_{\Theta_2} \label{conclusion}
\end{align}
for any $\Theta_1, \Theta_2\in \pi^{-1}([\Theta])$.
To prove this, we 
first note that 
\eqref{metric_M} yields 
\begin{align}
\langle {\rm D}\phi_U(\Theta)[\xi_1], {\rm D}\phi_U(\Theta)[\xi_2] \rangle_{\phi_U(\Theta)} 
= \langle \xi_1, \xi_2 \rangle_{\Theta} \label{touka}
\end{align}
for any $\xi_1, \xi_2\in T_{\Theta}N$ and any $U\in O(n)$,
where $\phi_U(\Theta):=U\circ \Theta$.
That is, the group action $\phi_U$ is an isometry in terms of Riemannian metric \eqref{metric_M}.
Eq.\,\eqref{touka} implies the following theorem.

\begin{theorem} \label{thm3}
For any $U\in O(n)$,
\begin{align}
\bar{\xi}_{\phi_U(\Theta)} = {\rm D}\phi_U(\Theta)[\bar{\xi}_{\Theta}]. \label{14}
\end{align}
\end{theorem}

We provide the proof of Theorem \ref{thm3} in Appendix \ref{apeA}.

Using Theorem \ref{thm3} and \eqref{touka}, we can prove \eqref{conclusion} as follows:
For any $\Theta_1, \Theta_2\in \pi^{-1}([\Theta])$, there exists some $U\in O(n)$ such that $\Theta_2 = \phi_U (\Theta_1)$.
Thus, 
\begin{align*}
\langle \bar{\xi}_{\Theta_2}, \bar{\zeta}_{\Theta_2} \rangle_{\Theta_2} &= \langle \bar{\xi}_{\phi_U(\Theta_1)}, \bar{\zeta}_{\phi_U(\Theta_1)} \rangle_{\phi_U(\Theta_1)}  \\
&= \langle {\rm D}\phi_U (\Theta_1)[\bar{\xi}_{\Theta_1}],  {\rm D}\phi_U (\Theta_1)[\bar{\zeta}_{\Theta_1}] \rangle_{\phi_U(\Theta_1)} \\
&= \langle \bar{\xi}_{\Theta_1}, \bar{\zeta}_{\Theta_1} \rangle_{\Theta_1},
\end{align*}
where the second and third equalities follow from \eqref{14} and \eqref{touka}, respectively.
Note that \eqref{conclusion} is based on isometric condition \eqref{touka} in terms of Riemannian metric \eqref{metric_M}.
In other words, if \eqref{touka} is not satisfied, we cannot conclude that \eqref{conclusion} holds.

Based on the above discussion, 
$N/O(n)$ endowed with \eqref{metric_quotient} is a Riemannian quotient manifold of $N$, and the natural projection $\pi: N\rightarrow N/O(n)$ becomes a Riemannian submersion.
That is, the projection $\pi$ is a smooth submersion, and for any $\Theta\in N$, the differential ${\rm D}\pi_{\Theta}$ is an isometry between the horizontal space $\mathcal{H}_{\Theta}$ and $T_{\pi(\Theta)}(N/O(n))$.
Moreover, \eqref{metric_quotient} is a unique Riemannian metric such that $\pi:N\rightarrow N/O(n)$ is a Riemannian submersion. 
This is
because, as shown in previously, $O(n)$ is a Lie group of isometries of the manifold $N$ endowed with \eqref{metric_M} that acts smoothly, freely, and properly on $N$.
For a more general description, see Proposition 2.28 in \cite{gallot1993riemannian}.
This means that if we introduce Riemannian metric \eqref{metric_M} into $N$,
the geometry of $N/O(n)$ is uniquely determined.

To summarize, Problem 2 is a Riemannian optimization problem, and most of the geometry of $N/O(n)$ can be studied by lifting from $N/O(n)$ to $N$.

\subsection{Problem 3}

Moreover, we can consider a simpler problem than Problems 1 and 2.
This is because,
for any $\Theta\in M$, there is a unique $\tilde{U}\in O(n)$ such that
\begin{align*}
\tilde{U}\circ \Theta = (\Lambda, \tilde{U}^{\top}B, C\tilde{U}),
\end{align*}
where $\Lambda\in {\rm Diag}_+(n)$.
That is, the above $\Theta$ and $\tilde{U}\circ\Theta$ realize input/output equivalent systems, i.e., $||e(\Theta)||_2=||e(\tilde{U}\circ\Theta)||_2$,
under the assumption that the initial state $\hat{x}_0$ is equal to zero.
The simpler problem is formulated as follows.

\begin{framed}
\begin{problem}
Suppose that the input/output data $\{ (u_0,y_0), (u_1,y_1),\ldots, (u_K,y_K)\}$ and state dimension $n$ are given.
Then, find the minimizer of
\begin{align*}
&{\rm minimize}\quad f_3(\Theta):= ||e(\Theta)||^2_2 \\
&{\rm subject\, to}\quad \Theta\in \tilde{M}.
\end{align*}
\end{problem}
\vspace{-.4\baselineskip}
\end{framed}

\noindent
Here, $\tilde{M}:={\rm Diag}_+(n)\times {\bf R}^{n\times m}\times {\bf R}^{p\times n}$.
However, we demonstrate in Section \ref{Sec6}  that, if the output data $y_0, y_1,\ldots, y_K$ are noisy,
the results provided by our algorithm for solving Problem 3 are more noise-sensitive than those produced by our algorithms for solving Problems 1 and 2.

Similar to Riemannian metric \eqref{metric_M} on $M$,
we define the Riemannian metric on $\tilde{M}$ as 
\begin{align}
&\quad \langle (\xi_1,\eta_1,\zeta_1),  (\xi_2,\eta_2,\zeta_2) \rangle_{\Theta}  \nonumber\\
&:= {\rm tr}(A^{-1}\xi_1A^{-1}\xi_2) + {\rm tr}(\eta_1^{\top}\eta_2) + {\rm tr}(\zeta_1^{\top} \zeta_2) \nonumber\\
&= {\rm tr}((A^{-1})^2\xi_1\xi_2) + {\rm tr}(\eta_1^{\top}\eta_2) + {\rm tr}(\zeta_1^{\top} \zeta_2)\label{metric_tildeM}
\end{align}
for $(\xi_1,\eta_1,\zeta_1),  (\xi_2,\eta_2,\zeta_2)\in T_{\Theta} \tilde{M}$.
Here, the second equality follows from the fact that $A^{-1}$, $\xi_1$, and $\xi_2$ are diagonal matrices.
Thus, Problem 3 is a Riemannian optimization problem.

\subsection{Another Riemannian metric on $M$, $N/O(n)$, and $\tilde{M}$}

Instead of Riemannian metric \eqref{metric_M}, we can introduce the Riemannian metric
\begin{align}
&\quad \langle (\xi_1,\eta_1,\zeta_1),  (\xi_2,\eta_2,\zeta_2) \rangle_{\Theta}  \nonumber\\
&:= {\rm tr}(\xi_1\xi_2) + {\rm tr}(\eta_1^{\top}\eta_2) + {\rm tr}(\zeta_1^{\top} \zeta_2) \label{metric_M2}
\end{align}
into the manifolds $M$ and $\tilde{M}$.
Moreover,
even if we define Riemannian metric \eqref{metric_M2} into $N$, a Riemannian metric on $N/O(n)$ can be defined by \eqref{metric_quotient}.
This is because \eqref{conclusion} holds under metric \eqref{metric_M2}.
That is,
in addition to Riemannian metric \eqref{metric_M},
Riemannian metric \eqref{metric_M2} also implies that the natural projection $\pi:N\rightarrow N/O(n)$ is a Riemannian submersion.
However, this simple metric is not adequate for solving Problems 1, 2, and 3,
as explained in Section \ref{sec4A}.

\begin{remark}
In this paper, we assume that the state dimension $n$ is given.
In practice, the dimension $n$ must be determined before solving Problems 1, 2, and 3.
For example, we can determine $n$ by using
Akaike's information criterion \cite{akaike1974new} or calculating the singular value decomposition of a matrix related to the input and output matrices \cite{verhaegen1994identification}.
\end{remark}

\begin{remark}
As mentioned in Section \ref{sec1}, we can identify $(F,G,H)$ in \eqref{1} using \eqref{F}, \eqref{G} and \eqref{H} after the identification of $(A,B,C)$ in \eqref{2}.
In addition to Problems 1, 2, and 3, we consider the following problem.
\begin{align*}
&{\rm minimize}\quad ||\exp(Fh)-A||^2_F \\
&{\rm subject\, to}\quad F\in {\rm Sym}(n).
\end{align*}
One may think that, by solving the above problem, $F\in {\rm Sym}(n)$ for \eqref{1} can be obtained even if $A\not\in {\rm Sym}_+(n)$.
However, this is not true.
For example, if $A=\begin{pmatrix}
0 & 0 \\
0 & 1
\end{pmatrix}$, then there is no solution $F$ to the above problem.
In fact, the infimum of the objective function is 0, whereas this value cannot be obtained with any $F\in {\rm Sym}(n)$.
\end{remark}

\begin{remark} \label{remark_future}
Note that system \eqref{1} does not correspond to a symmetric continuous-time system discussed in \cite{willems1976realization, tan2001stabilization}.
Here, system \eqref{1} is said to be symmetric in the sense of the definition in \cite{willems1976realization, tan2001stabilization} if there exists some $T\in GL(n)\cap {\rm Sym}(n)$ such that
$F^{\top} T = TF$ and
$H^{\top} = TG$.
\end{remark}


\section{Geometries of Problems 1, 2, and 3}  \label{Sec3}

\subsection{Riemannian optimization}

In preparation for subsequent subsections,
we introduce the concepts of the exponential mapping and the Riemannian gradient for Riemannian optimization \cite{absil2008optimization, edelman1998geometry, helmke1994optimization}, and provide a brief description of the optimization algorithm.
In this subsection, we consider a general Riemannian optimization problem of minimizing an objective function $f$ defined on a Riemannian manifold $\mathcal{M}$.
That is, $\mathcal{M}$ is equipped with a Riemannian metric $\langle \cdot, \cdot \rangle$ that endows the tangent space $T_x\mathcal{M}$ at each point $x\in \mathcal{M}$ with
an inner product.

Fig.\,\ref{Fig_opt_image} illustrates an optimization process on $\mathcal{M}$.
As shown in this figure, the next point is determined by using geodesics and search direction vectors.
The following explains the details.

\begin{figure}[t]
\begin{center}
\includegraphics[scale=0.26]{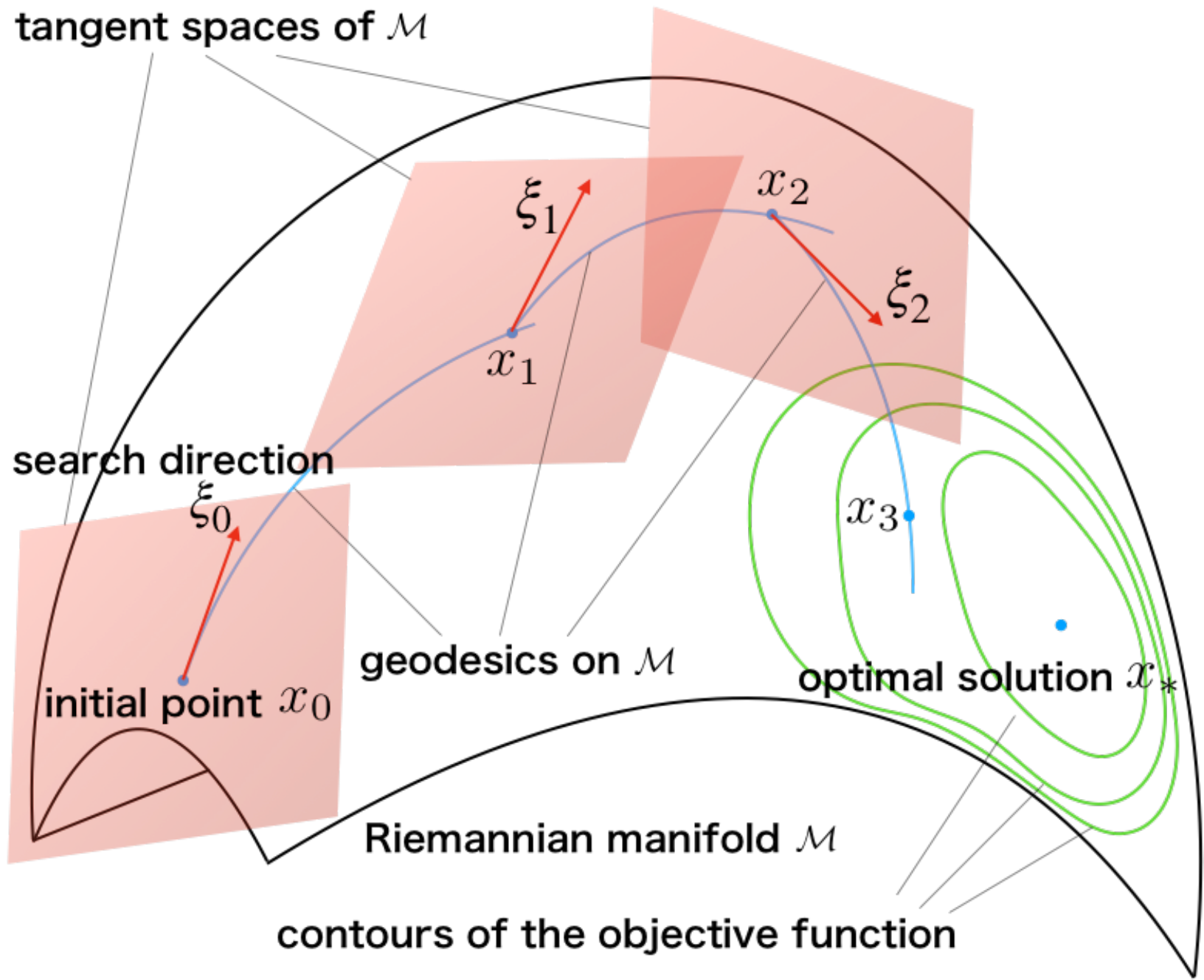}
\end{center}
\vspace{-1em}
\caption{Optimization process on a Riemannian manifold $\mathcal{M}$.} \label{Fig_opt_image}
\end{figure}

\subsubsection{Exponential mapping}
For the purpose of optimization on a Riemannian manifold $\mathcal{M}$, the update formula $x+\xi$ does not make sense for $x\in \mathcal{M}$ and $\xi\in T_x\mathcal{M}$.
This is in contrast to the case of optimization on a Euclidean space $\mathcal{E}$.
That is, on $\mathcal{E}$, we can compute a point $x_+\in \mathcal{E}$ from the current point $x\in \mathcal{E}$ and search direction $d\in \mathcal{E}$ as $x_+=x+d$.
Thus, we seek the next point $x_+$ on a curve called a geodesic on $\mathcal{M}$ emanating from $x$ in the direction of $\xi$.
For any $x,y\in \mathcal{M}$, on a geodesic between two points $x$ and $y$ that are sufficiently close,
the path along the geodesic is the shortest among all curves connecting $x$ and $y$.
It is known that, for any $\xi\in T_x\mathcal{M}$, there exists an interval $I\subset {\bf R}$ around $0$ and a unique geodesic $\Gamma_{(x,\xi)}:I\rightarrow \mathcal{M}$ such that
$\Gamma_{(x,\xi)}(0)=x$ and $\dot{\Gamma}_{(x,\xi)}(0) = \xi$.
The exponential mapping ${\rm Exp}$ at $x\in \mathcal{M}$ can be defined through the geodesic as
\begin{align*}
{\rm Exp}_x(\xi):= \Gamma_{(x,\xi)}(1),
\end{align*}
because the geodesic $\Gamma_{(x,\xi)}$ has the homogeneity property $\Gamma_{(x,a\xi)}(t) = \Gamma_{(x,\xi)}(at)$ for any $a\in {\bf R}$ satisfying $at\in I$. 

\subsubsection{Riemannian gradient}

In addition to the exponential mapping, we need a Riemannian gradient to solve our problems.
The Riemannian gradient ${\rm grad}\,f(x)$ of $f$ at $x\in \mathcal{M}$ is defined as a tangent vector at $x$ that satisfies
\begin{align*}
{\rm D}f(x)[\xi] = \langle {\rm grad}\,f(x), \xi\rangle_x
\end{align*}
for any $\xi\in T_x\mathcal{M}$, where ${\rm D}f(x)[\xi]$ is defined as
\begin{align*}
{\rm D}f(x)[\xi] := (\xi f)(x).
\end{align*}
Note that a tangent vector can be identified with a derivative.

\subsubsection{Algorithm}
The update formula of
a gradient algorithm for minimizing the objective function $f$ on $\mathcal{M}$ is given by
\begin{align*}
x_{k+1} = {\rm Exp}_{x_k}(\xi_k)
\end{align*}
with an initial point $x_0\in \mathcal{M}$,
where $\xi_k\in T_{x_k}\mathcal{M}$ is a search direction defined by using the Riemannian gradients.

\subsection{Geometry of Problem 1}

In first-order optimization algorithms such as the steepest descent and conjugate gradient methods on the manifold $M$ equipped with Riemannian metric \eqref{metric_M},
we need the Riemannian gradient of the objective function $f_1$.

Let $\bar{f}_1$ denote the extension of the objective function $f_1$ to the ambient Euclidean space ${\bf R}^{n\times n} \times {\bf R}^{n\times m} \times {\bf R}^{p\times n}$.
Then, the directional derivative of $\bar{f}_1$ at $\Theta\in M$ along $\xi=(\xi_A,\xi_B,\xi_C)\in T_{\Theta} M$ is given by
\begin{align}
{\rm D}\bar{f}_1(\Theta)[\xi] = 2({\rm D}e(\Theta)[\xi],e(\theta)), \label{Df}
\end{align}
where
\begin{align}
{\rm D}e(\Theta)[\xi] = \begin{pmatrix}
-{\rm D}\hat{y}_1(\Theta)[\xi] \\
-{\rm D}\hat{y}_2(\Theta)[\xi] \\
\vdots \\
-{\rm D}\hat{y}_K(\Theta)[\xi] 
\end{pmatrix}. \label{De}
\end{align}
Eq.\,\eqref{2} implies that 
\begin{align}
{\rm D}\hat{y}_k(\Theta)[\xi] = & C\sum_{i=0}^{k-1} A^{k-i-1} (\xi_A \hat{x}_i +\xi_B u_i) + \xi_C\hat{x}_k. \label{Dy}
\end{align}
It follows from \eqref{Df}, \eqref{De}, and \eqref{Dy} that
\begin{align}
{\rm D}\bar{f}_1(\Theta)[\xi] = {\rm tr}(\xi_A\, {\rm sym}(G_A)) + {\rm tr}(\xi_B^{\top} G_B) + {\rm tr}(\xi_C^{\top} G_C), \label{Df2}
\end{align}
where
\begin{align}
G_A &:= -2\sum_{k=1}^K \sum_{i=0}^{k-1} A^{k-i-1} C^{\top} (y_k-\hat{y}_k(\Theta)) \hat{x}_i^{\top}, \label{G_A}\\
G_B &:= -2\sum_{k=1}^K \sum_{i=0}^{k-1} A^{k-i-1} C^{\top} (y_k-\hat{y}_k(\Theta)) u_i^{\top}, \label{G_B}\\
G_C &:= -2\sum_{k=1}^K (y_k-\hat{y}_k(\Theta)) \hat{x}_k^{\top}. \label{G_C}
\end{align}
Here, we used the property $\xi_A=\xi_A^{\top}$.
Thus, the Euclidean gradient of $\bar{f}_1$ at $\Theta$ is given by
\begin{align}
\nabla \bar{f}_1(\Theta) = (G_A,G_B,G_C). \label{Euclid_grad}
\end{align}

\noindent
In \cite{sato2017riemannian}, we can find a similar derivation for a more complicated system. 
Because we introduced Riemannian metric \eqref{metric_M}, the Euclidean gradient in \eqref{Euclid_grad} yields  the Riemannian gradient
\begin{align}
{\rm grad}\, f_1(\Theta) = (A\,{\rm sym}(G_A)A, G_B, G_C). \label{grad1}
\end{align}
For a detailed explanation, see Appendix \ref{ape0}.

\subsection{Geometry of Problem 2} \label{Sec3C}

Because the natural projection $\pi:N\rightarrow N/O(n)$ is a Riemannian submersion,
most of the geometry of $N/O(n)$ can be studied by lifting from $N/O(n)$ to $N$,
as described in Section \ref{SecIIB}.

\subsubsection{Orthogonal projection onto the horizontal space $\mathcal{H}_{\Theta}$} \label{Sec3C2}

In Section \ref{Sec4B}, we need the concept of vector transport (which is a generalized concept of parallel transport \cite{absil2008optimization}) on the manifold $N/O(n)$ equipped with Riemannian metric \eqref{metric_quotient}
to develop a Riemannian conjugate gradient method. 
Here, note that we have introduced Riemannian metric \eqref{metric_M} into $N$.
To this end, for any $\Theta\in N$, we use the orthogonal projection $P_{\Theta}$ onto the horizontal space $\mathcal{H}_{\Theta}$.

To derive $P_{\Theta}$, we need to explicitly describe the vertical space $\mathcal{V}_{\Theta}$ and the horizontal space $\mathcal{H}_{\Theta}$.
First, we specify $\mathcal{V}_{\Theta}$.
Consider any curve $\Theta (t)$ on $\pi^{-1}([\Theta])\subset N$ with $\Theta(0)=\Theta$ that is expressed as
\begin{align*}
\Theta (t) = (U^{\top}(t) AU(t), U^{\top}(t)B,CU(t)),
\end{align*}
where $U(t)$ denotes a curve on $O(n)$ with $U(0)=I_n$.
Differentiating both sides with respect to $t$, we obtain
\begin{align*}
\dot{\Theta}(0) = (\dot{U}^{\top}(0)A+A\dot{U}(0), \dot{U}^{\top}(0)B, C\dot{U}(0)),
\end{align*}
where $\dot{U}(0)\in T_{I_n}O(n)\cong {\rm Skew}(n)$.
Thus, we have that
\begin{align*}
\mathcal{V}_{\Theta} = \{ (-U'A+AU', -U'B, CU')\, |\, U'\in {\rm Skew}(n) \}.
\end{align*}
Next, we characterize the horizontal space $\mathcal{H}_{\Theta}$.
Let $(A',B',C') \in\mathcal{H}_{\Theta}$.
That is,
\begin{align}
\langle (-U'A+AU', -U'B, CU'), (A',B',C') \rangle_{\Theta} = 0 \label{10}
\end{align}
for all $U'\in {\rm Skew}(n)$. This means that
\begin{align*}
{\rm tr}(U' ( 2A'A^{-1} +BB'^{\top}+C^{\top}C')) = 0.
\end{align*}
Because $U'\in {\rm Skew}(n)$ is arbitrary, we conclude that
\begin{align*}
2A'A^{-1} +BB'^{\top}+C^{\top}C' \in {\rm Sym}(n).
\end{align*}
That is, 
\begin{align*}
{\rm sk}(2A'A^{-1} +BB'^{\top}+C^{\top}C')=0.
\end{align*}
Thus, 
\begin{align*}
\mathcal{H}_{\Theta} \subset\{ (A',B',C')\,|\, {\rm sk}(2A'A^{-1} +BB'^{\top}+C^{\top}C')=0 \}.
\end{align*}
Conversely, if $(A',B',C')\in \{ (A',B',C')\,|\, {\rm sk}(2A'A^{-1} +BB'^{\top}+C^{\top}C')=0 \}$, we have that $(A',B',C')\in \mathcal{H}_{\Theta}$, because \eqref{10} holds.
Hence, we obtain
\begin{align}
\mathcal{H}_{\Theta} =\{ (A',B',C')\,|\, {\rm sk}(2A'A^{-1} +BB'^{\top}+C^{\top}C')=0 \}. \label{horizontal}
\end{align}

We are in a position to describe the orthogonal projection $P_{\Theta}$ onto the horizontal space $\mathcal{H}_{\Theta}$.
\begin{theorem} \label{Thm1}
The orthogonal projection $P_{\Theta}$ onto $\mathcal{H}_{\Theta}$ is given by 
\begin{align}
P_{\Theta}(\eta) = \eta +(XA-AX,XB,-CX), \label{OP}
\end{align}
where $\eta=(a,b,c)\in T_{\Theta} N$, and $X$ is the skew-symmetric solution to the linear matrix equation
\begin{align}
\mathcal{L}_1(X) +2\mathcal{L}_0(X) + \beta = 0, \label{linear}
\end{align}
where the linear matrix mappings $\mathcal{L}_0, \mathcal{L}_1: {\bf R}^{n\times n}\rightarrow {\bf R}^{n\times n}$ are defined by 
\begin{align*}
\mathcal{L}_0(X) &:= AXA^{-1}+A^{-1}XA -2X, \\
\mathcal{L}_1(X) &:= (BB^{\top}+C^{\top}C)X+X(BB^{\top}+C^{\top}C), 
\end{align*}
and $\beta := 2{\rm sk}(2A^{-1}a+bB^{\top}+c^{\top}C)$.
\end{theorem}

We provide the proof in Appendix \ref{Proof_Thm1}.

We can guarantee that there exists a unique solution $X\in {\rm Skew}(n)$ to \eqref{linear} under the assumption
\begin{align}
\dim ( {\rm Ker}\,(\lambda I_n-A) \cap {\rm Ker}\, B^{\top} \cap {\rm Ker}\, C) \leq 1\,\,\, {\rm for\,\,any}\,\,\lambda\in {\bf R}. \label{assumption}
\end{align}
Assumption \eqref{assumption} holds if matrix $A$ has only simple eigenvalues, because then $\dim({\rm Ker}\, (\lambda I_n-A))\leq 1$ for all $\lambda \in {\bf R}$.
Furthermore, if $(A,C)$ is observable, i.e., 
\begin{align*}
{\rm rank} \begin{pmatrix}
\lambda I_n -A \\
C
\end{pmatrix} = n \Leftrightarrow
{\rm Ker}\, (\lambda I_n-A) \cap {\rm Ker}\, C= \{ 0\}
\end{align*}
for all $\lambda\in {\bf C}$, then \eqref{assumption} holds.
Analogously, the controllability of $(A,B)$, i.e., 
\begin{align*}
{\rm rank} \begin{pmatrix}
\lambda I_n -A  & B
\end{pmatrix} =
{\rm rank} \begin{pmatrix}
\lambda I_n -A \\
B^{\top}
\end{pmatrix} = n,
\end{align*}
also implies \eqref{assumption}.

\begin{theorem} \label{Thm2}
Assume that \eqref{assumption} holds, and let $\mathcal{L}:= \mathcal{L}_1+2\mathcal{L}_0$.
Then, ${\rm Ker}\, \mathcal{L}= {\rm Ker}\, \mathcal{L}_1 \cap {\rm Ker}\, \mathcal{L}_0\subset {\rm Ker}\, \mathcal{L}_0 \subset {\rm Sym}(n)$.
In particular, $\mathcal{L}:{\rm Skew}(n)\rightarrow {\rm Skew}(n)$ is an automorphism. That is, for any $Y\in {\rm Skew}(n)$, there exists a unique $X\in {\rm Skew}(n)$ with
$\mathcal{L}(X)=Y$.
\end{theorem}

The proof is given in Appendix \ref{Proof_Thm2}.

\subsubsection{Riemannian gradient}

In numerical computations, we can use the horizontal lift
$\overline{{\rm grad}\, f_2}_{\Theta}$ as the Riemannian gradient at $[\Theta]\in N/O(n)$.
The horizontal lift belongs to the horizontal space $\mathcal{H}_{\Theta}$, and we have that
\begin{align}
\overline{{\rm grad}\, f_2}_{\Theta} = {\rm grad}\, f_1(\Theta), \label{grad2}
\end{align}
as shown in Section 3.6.2 in \cite{absil2008optimization}.
Thus, as the Riemannian gradient at $[\Theta]\in N/O(n)$, we can use ${\rm grad}\, f_1(\Theta)$, i.e., \eqref{grad1}.

\subsection{Geometry of Problem 3}

We have introduced Riemannian metric \eqref{metric_tildeM} into the manifold $\tilde{M}$.
Let $\bar{f_3}$ denote the extension of the objective function $f_3$ to the ambient Euclidean space ${\bf R}^{n\times n} \times {\bf R}^{n\times m} \times {\bf R}^{p\times n}$.
Then, the directional derivative of $\bar{f_3}$ at $\Theta\in \tilde{M}$ along $\xi=(\xi_A,\xi_B,\xi_C)\in T_{\Theta} \tilde{M}$ is given by
\begin{align}
{\rm D}\bar{f_3}(\Theta)[\xi] &=  {\rm tr}(\xi_A^{\top}\, G_A) + {\rm tr}(\xi_B^{\top} G_B) + {\rm tr}(\xi_C^{\top} G_C) \nonumber\\
&={\rm tr}(\xi_A\, {\rm diag}(G_A)) + {\rm tr}(\xi_B^{\top} G_B) + {\rm tr}(\xi_C^{\top} G_C), \label{Df3}
\end{align}
where $G_A$, $G_B$, and $G_C$ are defined by \eqref{G_A}, \eqref{G_B}, and \eqref{G_C}, respectively.
Here, we used the property that $\xi_A\in T_A {\rm Diag}_+(n)\cong {\rm Diag}(n)$.
Moreover, it follows from \eqref{metric_tildeM} and ${\rm D}f_3(\Theta)[\xi] = \langle {\rm grad}\,f_3(\Theta),\xi \rangle_{\Theta}$ that
\begin{align}
{\rm D}f_3(\Theta)[\xi] =& {\rm tr}((A^{-1})^2({\rm grad}\,f_3(\Theta))_A \xi_A) \nonumber \\
&+ {\rm tr}(\xi_B^{\top} ({\rm grad}\,f_3(\Theta))_B) + {\rm tr}(\xi_C^{\top} ({\rm grad}\,f_3(\Theta))_C). \label{Df4}
\end{align}
Because ${\rm D}f_3(\Theta)[\xi] = {\rm D}\bar{f}_3(\Theta)[\xi]$, \eqref{Df3} and \eqref{Df4} yield
\begin{align*}
{\rm grad}\,f_3(\Theta) = (A^2{\rm diag}(G_A), G_B, G_C).
\end{align*}

\section{Optimization algorithms for solving Problems 1, 2, and 3} \label{Sec4}

This section describes optimization algorithms for solving Problems 1,2, and 3,
and introduces a technique for choosing initial points in the algorithms.


\subsection{Optimization algorithm for solving Problem 1} \label{sec4A}

Algorithm \ref{algorithm1} describes a Riemannian CG method for solving Problem 1.
Because the Riemannian metric on the manifold $M$ is defined by \eqref{metric_M}, the exponential map ${\rm Exp}$ on $M$ is given by
\begin{align}
& {\rm Exp}_{\Theta} (A', B', C') \nonumber \\
=& (A^{1/2} \exp (A^{-1/2} A' A^{-1/2})A^{1/2}, B+B', C+C') \nonumber \\
=& (A\exp (A^{-1}A'), B+B', C+C'), \label{exp_alg1}
\end{align}
and the parallel transport $\mathcal{P}$ is given by
\begin{align}
& \mathcal{P}_{\Theta_1,\Theta_2}(A',B',C')  \nonumber \\
=& ((A_2A_1^{-1})^{1/2}A' ((A_2A_1^{-1})^{1/2})^{\top}, B', C'), \label{PT}
\end{align}
where $\Theta_i = (A_i,B_i,C_i)\in M$ $(i=1,2)$,
as shown in \cite{sra2015conic}.
We choose $t_k$ in step 4 as the Armijo step size \cite{absil2008optimization}.
The parameter $\beta_{k+1}$ in step 5 is called the Dai--Yuan type parameter \cite{sato2016dai}.

Note the if we introduce Riemannian metric \eqref{metric_M2}
instead of \eqref{metric_M},
exponential mapping \eqref{exp_alg1} is replaced with
\begin{align*}
{\rm Exp}_{\Theta} (A', B', C')
= (A+A', B+B', C+C').
\end{align*}
Thus, ${\rm Exp}_{\Theta} (A', B', C')\not \in M$ for some $(A',B',C')\in T_{\Theta} M$ because $A+A'$ is not always positive-definite.
As a result, we have to carefully choose $(A',B',C')\in T_{\Theta} M$ unlike for Riemannian metric \eqref{metric_M}.

The computational complexity of calculating the gradient ${\rm grad}\,f_1(\Theta)$ is higher than that of the other steps in Algorithm \ref{algorithm1}.
To estimate the complexity, we examine the complexities of $G_A$, $G_B$ and $G_C$.
To this end, we note that $G_A$ in \eqref{G_A} can be rewritten as
\begin{align*}
G_A = -2\sum_{i=0}^{K-1}\sum_{k=i+1}^K A^{k-i-1}C^{\top} (y_k-\hat{y}_k(\Theta)) \hat{x}_i^{\top}.
\end{align*}
Thus, we can recursively calculate $G_A$ as
\begin{align}
G_A(i+1) = G_A(i) -2 \gamma(i)\hat{x}^{\top}_{K-(i+1)}, \label{GA_efficient}
\end{align}
where
\begin{align*}
G_A(0) &= 0, \\
\gamma(i) &= C^{\top}(y_{K-i}-\hat{y}_{K-i}(\Theta)) + A\gamma(i-1), \\
\gamma(0) &= C^{\top}(y_K-\hat{y}_K(\Theta)).
\end{align*}
In fact, $G_A(K)=G_A$.
If $p<n$, i.e., the number of outputs is less than that of states, the computational complexity of $\gamma(i)\hat{x}^{\top}_{K-(i+1)}$ for each $i\in\{0,1,\ldots,K-1\}$ in \eqref{GA_efficient} is $\mathcal{O}(n^2)$,
because that of $\gamma(i)$ for each $i\in\{0,1,\ldots,K-1\}$ is $\mathcal{O}(n^2)$.
Thus, the complexity of $G_A$ is $\mathcal{O}(Kn^2)$.
Similarly, if $m<n$ and $p<n$, then the complexity of $G_B$ is $\mathcal{O}(Kn^2)$.
Moreover, if $p<n$, \eqref{G_C} implies that the complexity of $G_C$ is $\mathcal{O}(Kn^2)$. 
Hence, if $p,m<n<K$, \eqref{grad1} implies that the complexity of ${\rm grad}\,f_1(\Theta)$ is $\mathcal{O}(Kn^2)$.

\begin{algorithm}                      
\caption{Optimization algorithm for solving Problem 1.}         
\label{algorithm1}                          
\begin{algorithmic}[1]
\STATE Set input/output data $\{ (u_0,y_0), (u_1,y_1),\ldots, (u_K,y_K)\}$, the state dimension $n$, and an initial point $\Theta_0:=(A_0, B_0, C_0) \in M$.
\STATE Set $\eta_0 = -{\rm grad}\, f_1(\Theta_0)$ using \eqref{grad1}.
\FOR{$k=0,1,2,\ldots$ }
\STATE Compute a step size $t_k>0$, and set 
\begin{align}
 \Theta_{k+1}={\rm Exp}_{\Theta_k} (t_k \eta_{k}). \label{update_Alg1}
\end{align}
\STATE
Set
\begin{align*}
\beta_{k+1} = \frac{||g_{k+1}||_{k+1}^2}{ \langle g_{k+1},\mathcal{P}_{\Theta_k, \Theta_{k+1}}(\eta_k) \rangle_{k+1} -\langle g_k, \eta_k\rangle_k},
\end{align*}
where $g_k:= {\rm grad}\, f_1(\Theta_k)$, and $||\cdot||_k$ and $\langle \cdot, \cdot \rangle_k$ denote the norm and the inner product in the tangent space $T_{\Theta_k} M$, respectively.
\STATE Set
\begin{align}
\eta_{k+1} = -g_{k+1} + \beta_{k+1} \mathcal{P}_{\Theta_k, \Theta_{k+1}} (\eta_k). \label{eta_Alg1}
\end{align}
\ENDFOR
\end{algorithmic}
\end{algorithm}

\subsection{Optimization algorithm for solving Problem 2} \label{Sec4B}

In numerically solving Problem 2, we regard the manifold $M$ as $N$, because Proposition 3.3.12 in \cite{sontag1998mathematical} implies that the manifold $N$ is a dense set in the manifold $M$.
The proposed CG-based method for solving Problem 2 is obtained by replacing
the parallel transport $\mathcal{P}_{\Theta_k, \Theta_{k+1}}$ in Algorithm \ref{algorithm1} with  the orthogonal projection $P_{\Theta_{k+1}}$ given by \eqref{OP} onto the horizontal space $\mathcal{H}_{\Theta_{k+1}}$.
Note that the orthogonal projection $P_{\Theta_{k+1}}$ defines a vector transport on the quotient manifold $N/O(n)$ \cite{absil2008optimization}.


\subsection{Optimization algorithm for solving Problem 3}

The Riemannian CG method for solving Problem 3 is the same as Algorithm 1, except for the following:
\begin{itemize}
\item Replace $M$ with $\tilde{M}$.
\item Replace ${\rm grad}\, f_1(\Theta_k)$ with ${\rm grad}\, f_3(\Theta_k)$.
\end{itemize}
However, the computational complexity is lower than in the case of Problem 1.
This is because the matrices $A_k$ $(k=0,1,\ldots)$ in Algorithm 1 are diagonal when solving Problem 3, unlike for Problem 1.

\subsection{Initial points in Algorithm 1} \label{Sec4D}

To select an initial point $\Theta_0$ in Algorithm 1 for solving Problems 1 and 2, we propose Algorithm \ref{algorithm2} in which
${\rm rand}$ denotes a single uniformly distributed random number in the interval $(0,1)$.
In step 1, we obtain a triplet $(A,B,C)$ using
an existing subspace method such as N4SID \cite{van1994n4sid}, MOESP \cite{verhaegen1992subspace}, CVA \cite{larimore1990canonical}, ORT \cite{katayama2006subspace}, or N2SID \cite{verhaegen2016n2sid}.
However, at this stage, $A$ is not necessarily contained in ${\rm Sym}(n)$, to say nothing of ${\rm Sym}_+(n)$.
Thus, in step 2, we replace $A$ with the symmetric part of $A$.
That is, at this stage, $A\in {\rm Sym}(n)$, but $A\not\in {\rm Sym}_+(n)$ in general.
In fact, if there is $i\in \{1,2,\ldots, n\}$ such that $\lambda_i\leq 0$ in step 3, $A\not\in {\rm Sym}_+(n)$.
In steps $4, 5, 6, 7, 8$, any negative eigenvalues of $A$ are replaced with a random value in $(0,0.01)$.
That is, we consider negative eigenvalues of $A$ to be perturbed small positive eigenvalues.
Thus, steps 9 and 10 produce $A\in {\rm Sym}_+(n)$ and $\Theta_0\in M$, respectively.

For Problem 3, we replace step 9 in Algorithm \ref{algorithm2} with
\begin{align*}
A\leftarrow {\rm diag}(\lambda_1,\lambda_2,\ldots, \lambda_n),\,\,B\leftarrow V^{\top}B,\,\, C\leftarrow CV,
\end{align*}
where $V:=\begin{pmatrix}
v_1 & v_2 & \cdots & v_n
\end{pmatrix}$.
Then, $\Theta_0$ in step 10 is contained in $\tilde{M}$.

\begin{algorithm}                      
\caption{Constructing the initial point $\Theta_0\in M$.}         
\label{algorithm2}                          
\begin{algorithmic}[1]
\STATE Set $(A,B,C)$ using a subspace method.
\STATE $A\leftarrow {\rm sym}(A)$.
\STATE Let $A=\sum_{i=1}^n \lambda_i v_iv_i^{\top}$ be the eigenvalue decomposition. That is, $\lambda_i$ is an eigenvalue of $A$, and $v_i$ is the associated eigenvector.
\FOR{$i=1,2,\ldots, n$}
\IF{ $\lambda_{i} \leq 0$ }
\STATE $\lambda_i = 0.01\times {\rm rand}$.
\ENDIF
\ENDFOR
\STATE $A\leftarrow \sum_{i=1}^n \lambda_i v_iv_i^{\top}$.
\STATE $\Theta_0:=(A,B,C)$.
\end{algorithmic}
\end{algorithm}

\begin{remark} \label{remark2}
A Riemannian steepest descent (SD) method for solving Problems 1 and 2 can be derived by replacing steps 5 and 6 in Algorithm 1 with
$\eta_{k+1} = -{\rm grad}\,f_1(\Theta_{k+1})$.
That is, in contrast to the case of the CG methods, the SD method for Problem 2 is the same as that for Problem 1, because \eqref{grad2} holds.
However, the SD method is not more efficient than CG methods
\cite{ring2012optimization}.
We demonstrate this fact in Section \ref{Sec6A}.
\end{remark}


\section{GN method for solving the proposed problems} \label{Sec5}

The Gauss--Newton (GN) method has been widely used for solving least-squares problems.
Before comparing our proposed methods with the GN method, this section summarizes the GN method.

In the GN method \cite{mckelvey2004data, wills2008gradient},
we often use the vector parameter 
\begin{align}
\theta := \begin{pmatrix}
{\rm vec}_{\rm sym} (A) \\
{\rm vec}(B) \\
{\rm vec}(C)
\end{pmatrix}\in {\bf R}^{n_{\theta}} \label{full_parameter}
\end{align}
with 
$n_{\theta}:=n\left(\frac{n+1}{2} +m+p\right)$,
where ${\rm vec}$ denotes the usual vec-operator, i.e., ${\rm vec}(A)\in {\bf R}^{n^2}$ is obtained by stacking the columns of $A\in {\bf R}^{n\times n}$, and
for a symmetric matrix $A\in {\rm Sym}(n)$,
${\rm vec}_{\rm sym}(A)$ denotes the $\frac{1}{2}n(n+1)$-vector that is obtained from ${\rm vec}(A)$ by eliminating the redundant elements.
For example, if $A\in {\rm Sym}(3)$,
\begin{align*}
& {\rm vec}(A) \\
=& \begin{pmatrix}
a_{11} & a_{21} & a_{31} & a_{12} & a_{22} & a_{32} & a_{13} & a_{23} & a_{33}
\end{pmatrix}^{\top},\\
& {\rm vec}_{\rm sym}(A) \\
=& \begin{pmatrix}
a_{11} & a_{21} & a_{31} & a_{22} & a_{32} & a_{33}
\end{pmatrix}^{\top}.
\end{align*}

The parameter $\theta$ defined by \eqref{full_parameter} is a global coordinate system for the manifold $M$, and thus
we regard $\Theta$ on $M$ as $\theta$.
Hence, we write the prediction error vector $e(\Theta)$ defined by \eqref{pe_vector} as $e(\theta)$
and the objective function $f_1(\Theta)$ as
$V(\theta):= ||e(\theta)||_2^2$.
The aim of the GN method is to minimize $V(\theta)$.

The update formula of the GN method is given by
\begin{align}
\theta_{k+1} = \theta_k + t_k \Delta \theta_k, \label{GN_update}
\end{align}
where $t_k>0$ is a step size, and $\Delta \theta_k$ satisfies
\begin{align}
J(\theta_k)^{\top}J(\theta_k)\Delta \theta_k = -J(\theta_k)^{\top}e(\theta_k) \label{GN_equation}
\end{align}
with
\begin{align}
J(\theta) := \frac{\partial e}{\partial \theta}(\theta) \in {\bf R}^{pK\times n_{\theta}}. \label{J_def}
\end{align}
Here, 
\begin{align*}
\frac{\partial e}{\partial \theta_i}(\theta) = -\begin{pmatrix}
\left(\frac{\partial \hat{y}_1}{\partial \theta_i}(\theta)\right)^{\top} &
\left(\frac{\partial \hat{y}_2}{\partial \theta_i}(\theta)\right)^{\top} &
\cdots &
\left(\frac{\partial \hat{y}_K}{\partial \theta_i}(\theta)\right)^{\top}
\end{pmatrix}^{\top},
\end{align*}
and
\begin{align*}
\begin{cases}
\frac{\partial \hat{y}_j}{\partial \theta_i}(\theta) = \frac{\partial C}{\partial \theta_i}\hat{x}_j + C\frac{\partial \hat{x}_j}{\partial \theta_i}, \\
\frac{\partial \hat{x}_j}{\partial \theta_i} = \frac{\partial A}{\partial \theta_i}\hat{x}_{j-1} + A\frac{\partial \hat{x}_{j-1}}{\partial \theta_i} + \frac{\partial B}{\partial \theta_i}u_{j-1}
\end{cases}
\end{align*}
with
$\frac{\partial \hat{x}_0}{\partial \theta_i} = 0$.
Note that if $\theta_k \in N\subset M$ and the step size $t_k$ is sufficiently small, then $\Delta \theta_k$ can be regarded as an element of $T_{\theta_k} N$. 
In this case,
\begin{align}
e(\theta_k + t_k\Delta \theta_k) \approx e(\theta_k) + t_kJ(\theta_k)\Delta \theta_k, \label{e_approx}
\end{align}
and
\begin{align}
J(\theta_k): T_{\theta_k}N \rightarrow {\bf R}^{pK}. \label{J_map}
\end{align}

Note that \eqref{full_parameter} is an overparameterization.
This means that different $\theta$ may have the equivalent input-output properties.
In fact, from Section \ref{Sec3C}, each element $(A,B,C)$ of $\pi^{-1}([\Theta])\subset N$, which can be regarded as different $\theta$, has the same input-output properties, where
the dimension of $\pi^{-1}([\Theta])$ is
\begin{align}
\dim \pi^{-1}([\Theta]) = \dim N - \dim N/O(n) = \frac{n(n-1)}{2}. \label{dim_pi}
\end{align}
Eq.\,\eqref{dim_pi} follows from
Proposition 1 in Appendix \ref{ape_quotient} and Proposition 3.4.4 in \cite{absil2008optimization}.
Hence, if $\Delta \theta_k\in T_{\theta_k} \pi^{-1}([\Theta_k]) \subset T_{\theta_k} N$ under the identification of $\Theta_k$ and $\theta_k$,
it follows from \eqref{e_approx} that $J(\theta_k)\Delta \theta_k=0$, that is, $\Delta \theta_k \in {\rm Ker}\, J(\theta_k)$.
Therefore, 
\begin{align*}
T_{\theta_k}\pi^{-1}([\Theta_k]) \subset {\rm Ker}\, J(\theta_k),
\end{align*}
and \eqref{dim_pi} yields
\begin{align}
\dim {\rm Ker}\, J(\theta_k)\geq \frac{n(n-1)}{2}. \label{Ker_J}
\end{align}
It follows from \eqref{Ker_J} that the matrix $J(\theta_k)$ is rank-deficient, and
thus there are infinitely many solutions to \eqref{GN_equation}.
In the data-driven local coordinates (DDLC) introduced in \cite{mckelvey2004data}, $\Delta \theta_k$ is chosen as
\begin{align}
\Delta \theta_k = -V_1S_1^{-1}U_1^{\top} e(\theta_k) \label{DDLC}
\end{align}
as shown in \cite{wills2008gradient},
where $U_1S_1V_1^{\top}$ is the truncated singular value decomposition of $J(\theta_k)$, and $S_1\in {\bf R}^{n_{\theta}\times n_{\theta}}$ is a diagonal matrix.

Note that update formula \eqref{GN_update} preserves the symmetry of $A$, but does not, in general, preserve the positive-definiteness.
More precisely, if $\theta_k$ is contained in the manifold $M$ or $N$,
$\theta_{k+1}$ given by \eqref{GN_update} is also contained in $M$ or $N$ by choosing sufficiently small $t_k>0$.
This is because $M$ and $N$ are open sets.
However, if $t_k>0$ is too small, the objective function $V$ does not change very much.
Thus, we need to choose sufficiently large $t_k>0$, but then $\theta_{k+1}$ may not be contained in $M$ nor $N$, as demonstrated in Section \ref{Sec6}.
That is, it is difficult to determine an appropriate step size $t_k$ for some examples.

Instead of using the full parameterization $\theta\in {\bf R}^{n_\theta}$ defined by \eqref{full_parameter},
we can use canonical forms of linear systems, which reduce the number of free parameters.
However, canonical forms may lead to numerically ill-conditioned problems due to noise, as pointed out in Sections 1 and 4 in \cite{mckelvey2004data}.
Moreover, the results of Section \ref{Sec6B1} justify this point, because Problem 3 can be regarded as a case of using a canonical form.
Thus, in practice, the use of canonical forms for system identification may not be adequate.

To resolve the numerically ill-conditioned problem,
the use of overlapping parameterization has been proposed in \cite{van1982line}, and the authors in \cite{hanzon1998overlapping} introduced block-balanced input normal forms, which are overlapping parameterizations.
However, this approach requires monitoring the condition of the parametrization and switching to a new structure if the current structure is bad.
That is, this needs a number of extra calculations, which are not necessary in the case of Algorithm 1 based on Riemannian optimization.

\begin{remark}
Although the Jacobian $J(\theta)$ is rank-deficient in our case,
if $J(\theta)$ is of full-rank, update formula \eqref{GN_update} for the GN method can be regarded as a 
Riemannian steepest descent method for the specific choice of the 
Riemannian metric
\begin{align}
g^{{\rm GN}}_{\theta}(\dot{\theta}_1,\dot{\theta}_2) := \dot{\theta}_1^{\top} R(\theta) \dot{\theta}_2 \label{GN_metric1}
\end{align}
with
$R(\theta):= 2J(\theta)^{\top} J(\theta)$
into ${\bf R}^{n_{\theta}}$, as stated in \cite{hanzon1993riemannian, peeters1994system}.
The Riemannian gradient of the objective function $V$ at $\theta\in {\bf R}^{n_{\theta}}$ is given by
$R(\theta)^{-1}\frac{\partial V}{\partial \theta}(\theta)$,
where the gradient is called the natural gradient in the Riemannian manifold ${\bf R}^{n_{\theta}}$ endowed with the Riemannian metric \eqref{GN_metric1} \cite{amari1998natural}.
Using the natural gradient, update formula \eqref{GN_update} can be expressed as
\begin{align}
\theta_{k+1} = \theta_k -R(\theta_k)^{-1}\frac{\partial V}{\partial \theta}(\theta_k). \label{GN_update2}
\end{align}
\end{remark}

\section{Numerical Simulations} \label{Sec6}

In this section, we demonstrate the effectiveness of the proposed method.
To this end, we evaluate the identified systems using various indices, in addition to the value of the objective function in Problems 1, 2, and 3,
to prevent overfitting to noisy data.
Note that, in the simulations, we used MOESP \cite{verhaegen1992subspace} as the subspace method for step 1 in Algorithm \ref{algorithm2}.
This was implemented in the {\it system identification toolbox} of MATLAB.
Hence, we can easily implement Algorithm 2.

We consider identification problems of the RC electrical network system \cite{dorfler2018electrical, vandenberghe1997optimal} represented as
the undirected graph $\mathcal{G}=\{ \{1,2,\ldots, n\}, \mathcal{E} \}$, which is composed of $n$ nodes and the set $\mathcal{E}$ of $k$ undirected edges.
A mathematical model of the system is described by
\begin{align}
\begin{cases}
C_{\rm cap}\dot{V}(t) = -(\mathcal{L}_{{\rm res}} + G_{\rm con})V(t) + \tilde{G} u(t), \\
y(t) = \tilde{H}V(t).
\end{cases} \label{true}
\end{align}

\noindent
Table \ref{table_parameter} explains the parameters of system \eqref{true}.
Here, $\mathcal{L}_{\rm res}:= \mathcal{B}R_{\rm res}^{-1}\mathcal{B}^{\top}$ is a symmetric positive semi-definite matrix, 
 $\mathcal{B}\in {\bf R}^{n\times k}$ is the incidence matrix of $\mathcal{G}$, and
 $R_{\rm res}\in {\rm Diag}_+(n)$ is the resistance matrix.
The incidence matrix $\mathcal{B}=(\mathcal{B}_{ij})\in {\bf R}^{n\times k}$ is defined by
\begin{align*}
\mathcal{B}_{ij} := \begin{cases}
1,\quad\,\,\,\,\, {\rm if}\,\, i\,\, {\rm is\,\, the\,\, source\,\, node\,\, of\,\, edge}\,\, j \\
-1,\quad {\rm if}\,\, i\,\, {\rm is\,\, the\,\, sink\,\, node\,\, of\,\, edge}\,\, j \\
0,\quad\,\,\,\,\, {\rm otherwise}.
\end{cases}
\end{align*}
System \eqref{true} can be transformed into \eqref{1}
by defining $x(t):=C_{\rm cap}^{1/2}V(t)$:
\begin{align}
\begin{cases}
\dot{x}(t) = Fx(t)+Gu(t),\\
y(t) = Hx(t)
\end{cases} \label{numerical_ex}
\end{align} 
with
\begin{align*}
F &:=  -C_{\rm cap}^{-1/2}(\mathcal{L}_{{\rm res}} + G_{\rm con})C_{\rm cap}^{-1/2}\in {\bf R}^{n\times n},\\
G &:= C_{\rm cap}^{-1/2}\tilde{G}\in {\bf R}^{n\times m},\\
H &:= \tilde{H}C_{\rm cap}^{1/2} \in {\bf R}^{p\times n}.
\end{align*}
Note that the matrix $-F$ is contained in ${\rm Sym}_+(n)$, and thus $F$ is stable.
That is, all the eigenvalues of $F$ are negative.

\begin{table}[t]
\caption{Parameters of system \eqref{true}.} \label{table_parameter}
  \begin{center}
    \begin{tabular}{c|c} \hline
       $V(t)\in {\bf R}^n$    &  node-voltage vector    \\ \hline
$u(t)\in {\bf R}^m$ & voltage-source vector \\ \hline  
$y(t)\in {\bf R}^p$ &  measurement output vector \\ \hline
$C_{\rm cap}\in {\rm Sym}_+(n)$ &  capacitance matrix \\ \hline
$G_{\rm con}\in {\rm Sym}_+(n)$ &  conductance matrix \\ \hline
$\mathcal{L}_{\rm res}\in {\rm Sym}(n)$ & Laplacian matrix associated with $\mathcal{G}$ \\ \hline
$\tilde{G}\in {\bf R}^{n\times m}$ and $\tilde{H}\in {\bf R}^{p\times n}$  & constant matrices \\ \hline
   \end{tabular}
  \end{center}
\end{table}

Although we consider mathematical model \eqref{numerical_ex} to be noise-free,
measurement noise is inevitable in practice, as explained in Section 4.3 in \cite{ljung1999system}.
Thus, we assume that the true system is given by
\begin{align}
\begin{cases}
x_{k+1} = Ax_k+Bu_k,\\
y_k = Cx_k + v_k,
\end{cases} \label{numerical_ex2}
\end{align}
where $A$, $B$, and $C$ are defined by \eqref{3}, \eqref{4}, and \eqref{5}, respectively, and $v_k\in {\bf R}^p$ is measurement noise.
That is,
 the input/output data $(u_k,y_k)$ is generated by \eqref{numerical_ex2}.
Because $F$ is stable, the matrix $A$ is also stable. 
That is, all eigenvalues of $A$ are in the interval $(0,1)$.
The signal-to-noise ratio of system \eqref{numerical_ex2} is defined as
\begin{align}
{\rm SNR} = 10\log_{10} \left( \frac{\sum_{k=0}^K ||y_k-v_k||^2_2}{\sum_{k=0}^K ||v_k||^2_2} \right). \label{snr}
\end{align}

In the following, we present the results of numerical simulations for SISO and MIMO cases.
For SISO cases, we illustrate a frequency response using the Bode plots.
For MIMO cases, the values of various indices are given, because 
Bode plots  of MIMO cases do not clarify the distance between the true and estimated systems.

To this end, we set $n=20$, and
generated the undirected graph $\mathcal{G}$ using the Watts and Strogatz model \cite{watts1998collective} with 20 nodes of mean degree 10  and rewiring probability 0.4.
Additionally, $C_{\rm cap}$, $R_{\rm res}$, and $G_{\rm con}$ were given by
\begin{align}
\begin{cases}
C_{\rm cap} = 10\times {\rm diag}({\rm rand}, {\rm rand},\ldots, {\rm rand}), \\
R_{\rm res} = 0.1\times {\rm diag}(1, 1, \ldots, 1), \\
G_{\rm con} = {\rm diag}({\rm rand}, {\rm rand}, \ldots, {\rm rand}),
\end{cases} \label{CRG}
\end{align}
where each ${\rm rand}$ denotes a uniformly distributed random number in the interval $(0,1)$.
Moreover, we generated each component of $u_k$ from the Gaussian random distribution with mean $0$ and variance 100, and
the components of $v_k$ from the Gaussian random distribution with mean 0 and variance $\sigma^2$.
The sampling interval $h$ was $0.1$.

We denote the results given by Algorithm 1 for solving Problems 1, 2, and 3 as ${\rm CG}_1$, ${\rm CG}_2$, and ${\rm CG}_3$, respectively.
Moreover, we write ${\rm SD}$ to denote the Riemannian SD method, as briefly explained in Remark \ref{remark2}.

\subsection{SISO case} \label{Sec6A}

First, we considered SISO cases with $m=p=1$.
The parameters $\tilde{G}$ and $\tilde{H}$ were given by
\begin{align*}
\tilde{G} = \begin{pmatrix}
1 \\
0 \\
\vdots \\
0
\end{pmatrix},\quad
\tilde{H} = \begin{pmatrix}
1 & 0 &\cdots & 0
\end{pmatrix}.
\end{align*}

\subsubsection{Identification by the GN method}

Fig.\,\ref{Fig1} illustrates the eigenvalues of the true matrix $A$ corresponding to $F$ of system \eqref{numerical_ex}, the estimated matrix $A$ produced by Algorithm 2,
and the estimated matrix $A$ provided by the prediction error method using the GN method with the update formula \eqref{GN_update} and \eqref{DDLC}, as explained in Section \ref{Sec5}, after 10 iterations.
Here, we used the result $(A,B,C)$ obtained by Algorithm 2 as the initial point of the  GN method, and the step sizes $t_k$ in \eqref{GN_update} were $t_k=10^{-9}$ for all $k\in \{1,2,\ldots, 10\}$.
According to Fig.\,\ref{Fig1}, the prediction error method using the GN method did not provide $\Theta\in M$.
In fact, some eigenvalues of $A$ produced by the GN method took negative values, whereas all eigenvalues of the true matrix $A$ are positive. 
Moreover, we confirmed the following results:
\begin{itemize}
\item When $t_k=10^{-9}$, the positive-definite property of matrix $A$ produced by the GN method was lost after only a few iterations.
\item If we set $t_k>10^{-9}$, the symmetric matrices $A$ produced by the GN method, in many cases, were unstable after 10 iterations.
\item Even if we set $t_k<10^{-9}$, some eigenvalues of the symmetric matrices $A$ produced by the GN method were negative after 10 iterations.
\end{itemize}


Thus, the GN method described in Section \ref{Sec5} is not adequate for solving our problem.
Hence, we hereafter compare ${\rm CG}_1$, ${\rm CG}_2$, ${\rm CG}_3$,  {\rm SD}, and Algorithm 2.

\begin{figure}[t]
\begin{center}
\includegraphics[scale=0.44]{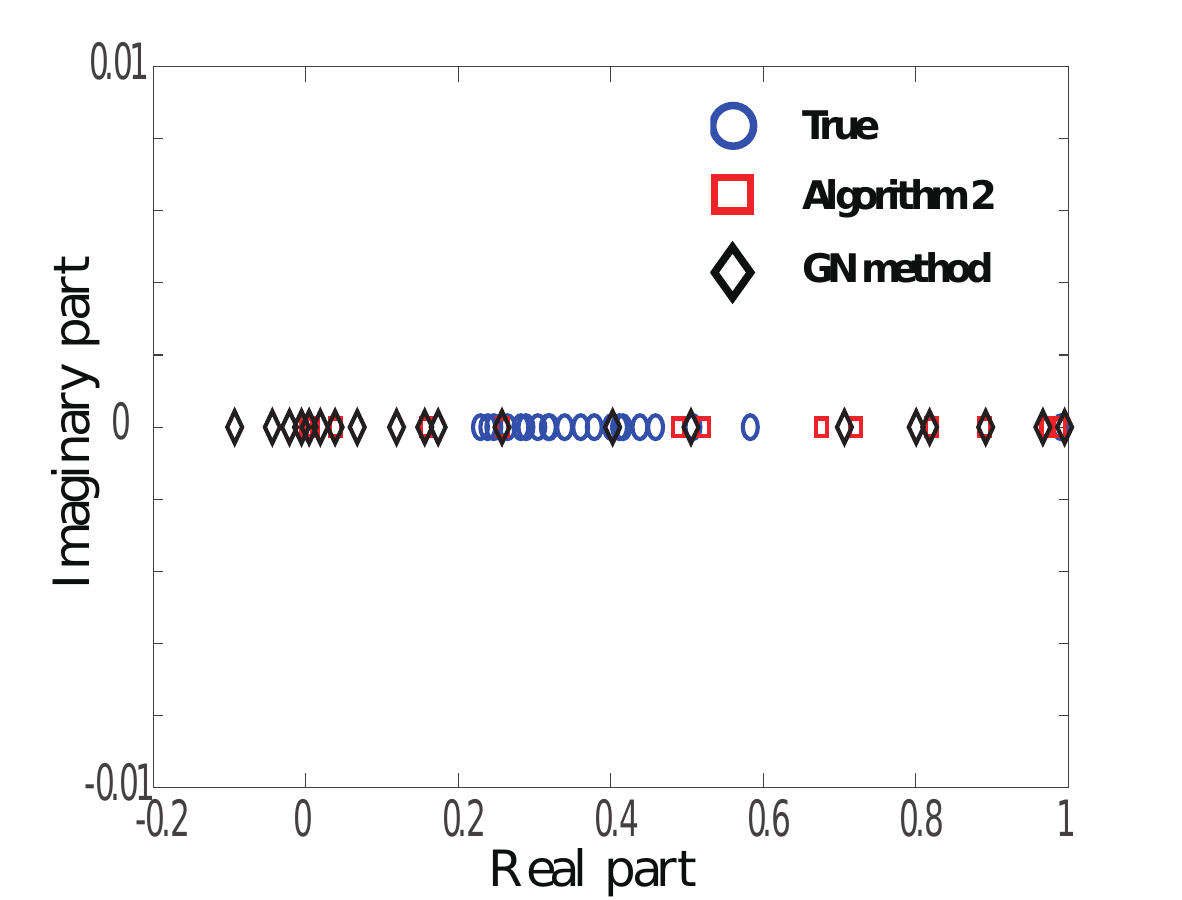}
\end{center}
\vspace{-1em}
\caption{Eigenvalues of $A$ in the original system, estimated system produced by Algorithm 2, and estimated system provided by the GN method after 10 iterations.} \label{Fig1}
\end{figure}


\subsubsection{Comparison of ${\rm CG}_1$, ${\rm CG}_2$, ${\rm CG}_3$, and {\rm SD}}

Figs.\,\ref{Fig2} and \ref{Fig3} illustrate a comparison of ${\rm CG}_1$, ${\rm CG}_2$, ${\rm CG}_3$, and SD with $K=600$, $\sigma^2=0.1$, and ${\rm SNR}=12.803$.
Here, $\Theta_0$ in Fig.\,\ref{Fig2} was obtained using Algorithm 2.
According to these figures, the results for ${\rm CG}_1$ and ${\rm CG}_2$ completely overlap,
and Fig.\,\ref{Fig2} demonstrates that the convergence speeds of ${\rm CG}_1$, ${\rm CG}_2$, and ${\rm CG}_3$ are superior to that of SD. 
Moreover, Fig.\,\ref{Fig3} shows that ${\rm CG}_1$, ${\rm CG}_2$, ${\rm CG}_3$, and {\rm SD} improve the frequency response of Algorithm 2.
In particular, the Bode plots of the estimated systems obtained by ${\rm CG}_1$ and ${\rm CG}_2$ are almost the same as that of the true system, unlike ${\rm CG}_3$, SD, and Algorithm 2.
Note that no destabilization occurred for ${\rm CG}_1$, ${\rm CG}_2$, or ${\rm CG}_3$.

\begin{figure}[t]
\begin{center}
\includegraphics[scale=0.44]{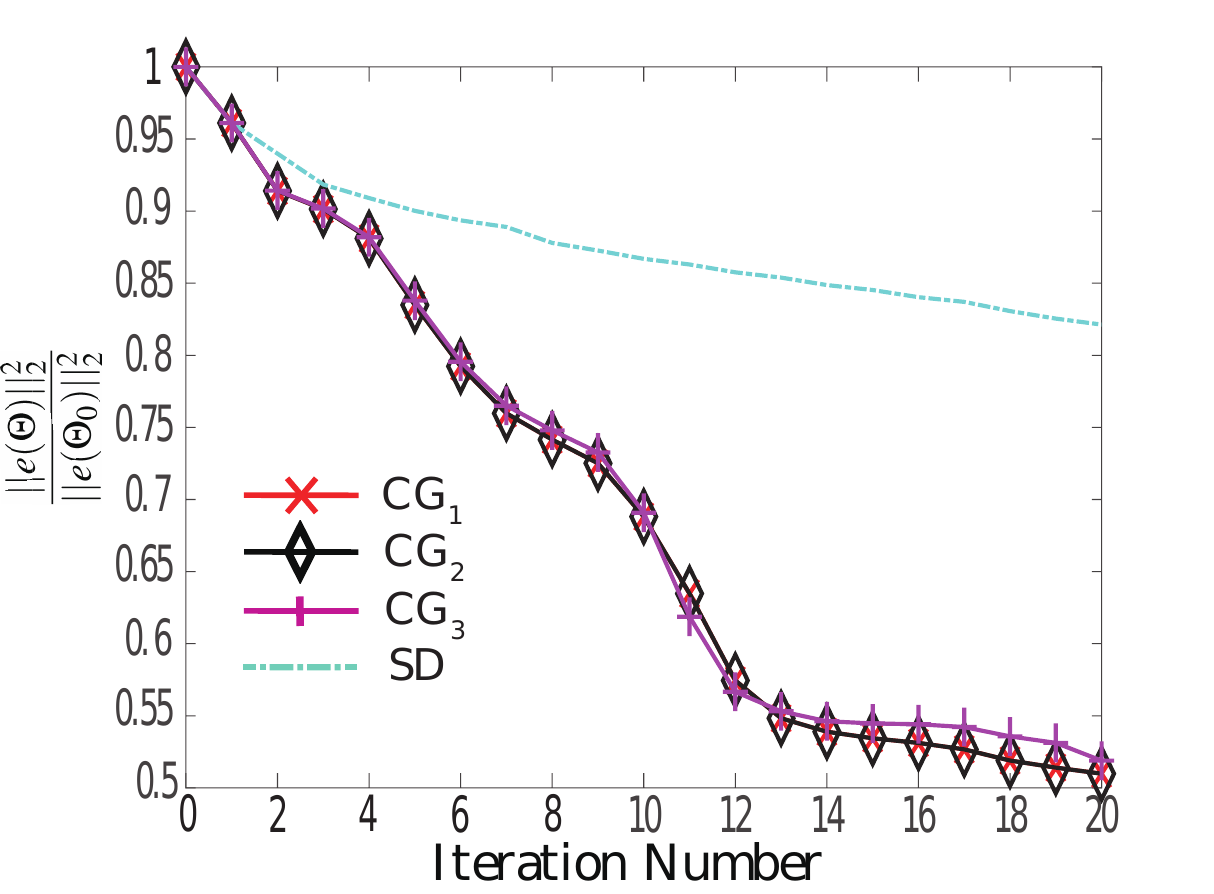}
\end{center}
\vspace{-1em}
\caption{Relative objective values obtained by ${\rm CG}_1$, ${\rm CG}_2$, ${\rm CG}_3$, and {\rm SD}.} \label{Fig2}
\begin{center}
\includegraphics[scale=0.44]{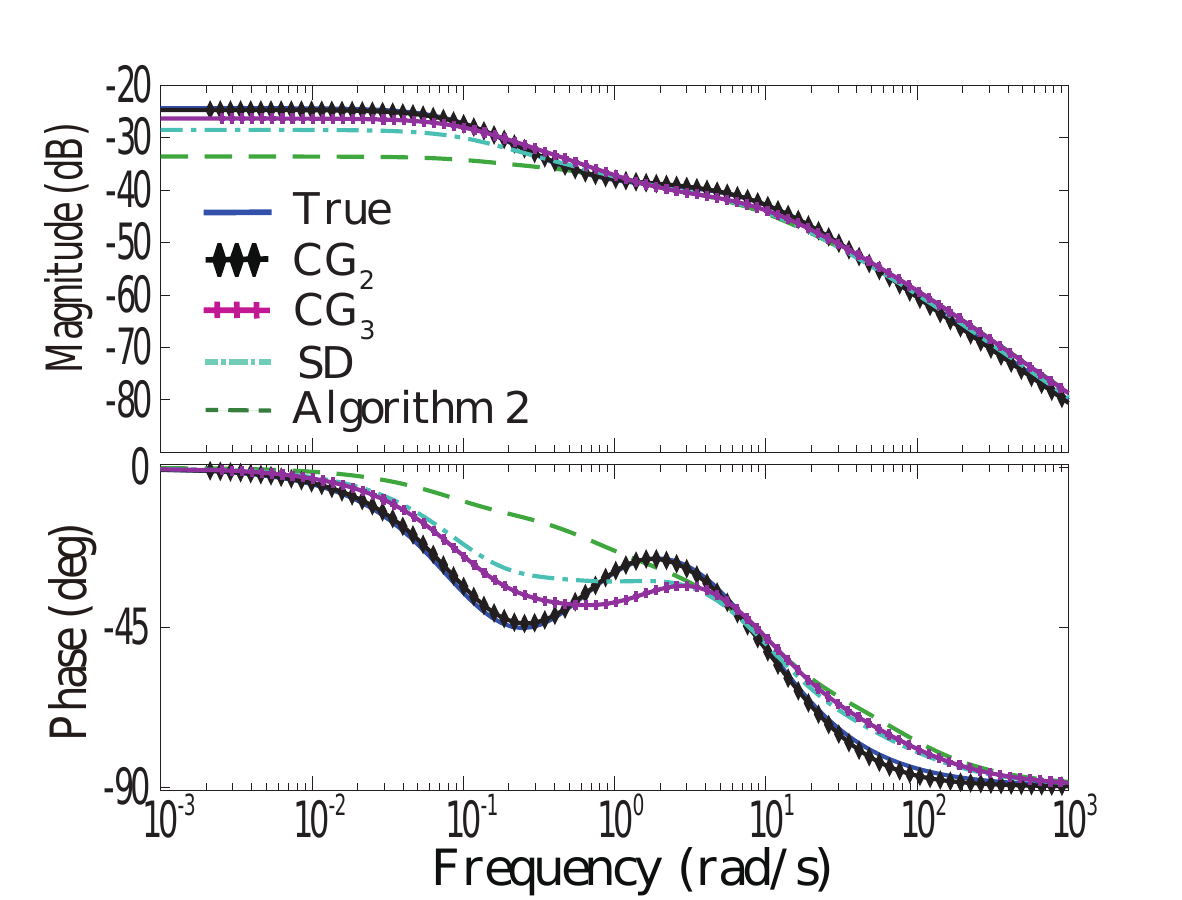}
\end{center}
\vspace{-1em}
\caption{Bode plots of true and estimated systems obtained by ${\rm CG}_1$, ${\rm CG}_2$, ${\rm CG}_3$, {\rm SD}, and Algorithm 2. Because the Bode plot of the estimated system obtained by ${\rm CG}_1$ completely overlapped with that obtained by ${\rm CG}_2$, the illustration of ${\rm CG}_1$ was omitted.} \label{Fig3}
\end{figure}

\subsection{MIMO case}

Next, we considered the MIMO case with $m=p=2$.
The parameters $\tilde{G}$ and $\tilde{H}$ were given by
\begin{align*}
\tilde{G} = \begin{pmatrix}
1 & 0 \\
0 & 1 \\
0 & 0 \\
\vdots & \vdots \\
0 & 0
\end{pmatrix},\quad
\tilde{H} = \begin{pmatrix}
1 & 0 & 0 & \cdots & 0\\
0 & 1 & 0 & \cdots & 0
\end{pmatrix}.
\end{align*}
As with the SISO case, 
the conventional GN method did not produce $A_{\rm est}\in {\rm Sym}_+(n)$ and the convergence speeds of ${\rm CG}_1$, ${\rm CG}_2$, and ${\rm CG}_3$ were faster than that of SD.
Thus, we present the results of comparisons among ${\rm CG}_1$, ${\rm CG}_2$, ${\rm CG}_3$, and Algorithm \ref{algorithm2}.

\subsubsection{Stability of the estimated matrices $A\in {\rm Sym}_+(n)$ produced by ${\rm CG}_1$, ${\rm CG}_2$, and ${\rm CG}_3$} \label{Sec6B1}

Because all of the matrices $A_{\rm est}$ (estimates of $A$ in true system \eqref{numerical_ex2}) produced by ${\rm CG}_1$, ${\rm CG}_2$, and ${\rm CG}_3$ are contained in ${\rm Sym}_+(n)$,
the eigenvalues of $A_{\rm est}$ are positive real numbers, unlike the eigenvalues given by the conventional GN method.
However, even if $A$ is stable, $A_{\rm est}$ may be unstable.

Thus, we compared the stability of the estimated matrices $A_{\rm est}$ provided by ${\rm CG}_1$, ${\rm CG}_2$, and ${\rm CG}_3$.
We performed numerical simulations 100 times with $\sigma^2=0.05$, $\sigma^2=0.1$, and $\sigma^2=0.5$.
Table \ref{table1} presents the number of unstable cases over 20 iterations when $K=400$.
According to Table \ref{table1}, the rate of instability in $A_{\rm est}$ produced by ${\rm CG}_3$ is far higher than when using ${\rm CG}_1$ or ${\rm CG}_2$. 
Because we used different $C_{\rm cap}$, $R_{\rm res}$, and $G_{\rm con}$ for each simulation, the {\rm SNR} defined by \eqref{snr} was also different.
Table \ref{table4} describes the relation between $\sigma^2$ and ${\rm SNR}$.
Here, ${\rm SNR}_{\rm ave}$ and ${\rm SNR}_{\rm dev}$ are the average and standard deviation over 10000 simulations, defined by
\begin{align*}
{\rm SNR}_{\rm ave} &:= \frac{ \sum_{i=1}^{10000}{\rm SNR}_i }{10000}, \\
{\rm SNR}_{\rm dev} &:= \sqrt{\frac{ \sum_{i=1}^{10000}({\rm SNR}_i-{\rm SNR}_{\rm ave})^2 }{10000} },
\end{align*}
where ${\rm SNR}_i$ denotes ${\rm SNR}$ in the $i$-th simulation.
According to Table \ref{table4}, ${\rm SNR}_{\rm ave}$ decreases as $\sigma^2$ increases, although the ${\rm SNR}_{\rm dev}$ values are similar.
We also obtained similar results to those described in Tables \ref{table1} and \ref{table4} for different values of $K$.
Hence, we conclude that the rate of instability in $A_{\rm est}$ produced by
${\rm CG}_3$ is far higher than those when using ${\rm CG}_1$ and ${\rm CG}_2$.
This is in contrast to the SISO case.
In addition, the instability rate for ${\rm CG}_1$ and ${\rm CG}_2$ is independent of SNR, unlike that for ${\rm CG}_3$.

The reason for the high instability rate produced by ${\rm CG}_3$ is that the noise component of the output directly influences
 the diagonal matrix $A_{\rm est}$, i.e.,
 eigenvalues of $A_{\rm est}$.
This is essentially the same phenomenon observed in system identification problems,
whereby canonical forms lead to numerically ill-conditioned problems \cite{mckelvey2004data}.
In contrast, the noise component does not have a significant effect on the eigenvalues of the estimated matrices $A_{\rm est}$ produced by ${\rm CG}_1$ and ${\rm CG}_2$,
because the matrices are not diagonal.

\begin{table}[t]
\caption{Number of unstable cases over $20$ iterations when $K=400$.} \label{table1}
  \begin{center}
    \begin{tabular}{c|ccc} \hline
          Algorithm                           &  $\sigma^2=0.05$ &   $\sigma^2=0.1$  & $\sigma^2=0.5$   \\ \hline \hline \vspace{3mm}
${\rm CG}_1$ & $3$  &  $2$ & $3$ \\ \hline  
${\rm CG}_2$ & $1$   & $1$ & $2$ \\ \hline
${\rm CG}_3$ & $15$  & $17$ & $26$ \\ \hline
   \end{tabular}
  \end{center}
\end{table}

\begin{table}[t]
\caption{${\rm SNR}_{\rm ave}$ and ${\rm SNR}_{\rm dev}$   when $K=400$.} \label{table4}
  \begin{center}
    \begin{tabular}{c|ccc} \hline
                                     &  $\sigma^2=0.05$ &   $\sigma^2=0.1$  & $\sigma^2=0.5$   \\ \hline \hline \vspace{3mm}
${\rm SNR}_{\rm ave}$ & 26.261  &  20.174 & 6.212 \\ \hline  
${\rm SNR}_{\rm dev}$ & 6.686   & 6.660 & 6.686 \\ \hline
   \end{tabular}
  \end{center}
\end{table}

\subsubsection{Evaluation of proposed methods} \label{Sec6B2}

We evaluated the results with respect to the cost function $||e(\Theta_{\rm est})||_2^2$, the relative $H^2$ and $H^{\infty}$ norms,
and the maximum eigenvalues $\lambda_{\rm max}(F_{\rm est})$ of the estimated matrix $F_{\rm est}$ of $F$.
Here, $\lambda_{\rm max}(F)$ was $-0.086$ in all cases. 
Note that the maximum eigenvalue $\lambda_{\rm max}(F_{\rm est})$ is important,
because the transient state $\hat{x}(t)$ in system \eqref{1} is dominated by $\lambda_{\rm max}(F)$ under $\hat{u}(t)=0$.
That is, if $\lambda_{\rm max}(F)$ and $\lambda_{\rm max}(F_{\rm est})$ are closer, we can expect the true and estimated transient states to be more similar. 
When we used our proposed methods ${\rm CG}_1$, ${\rm CG}_2$, and ${\rm CG}_3$, the number of iterations was set to 20.
Increasing the number of iterations would decrease the value of the objective function $||e(\Theta_{\rm est})||_2^2$.
However, other indices such as $g_2$, $g_{\infty}$, and $\lambda_{\rm max}(F_{\rm est})$ may become worse due to overfitting with noisy data.

To define the relative $H^2$ and $H^{\infty}$ norms, we use
$T$ and $T_{\rm est}$ as the transfer functions from the input $u$ to the output $y$ of the true and estimated systems, respectively.
That is, 
\begin{align*}
T(s) &:= C(sI_n-F)^{-1}G,\quad s\in {\bf C}, \\
T_{\rm est}(s) &:= C_{\rm est}(sI_n-F_{\rm est})^{-1}G_{\rm est},
\end{align*}
where $G_{\rm est}$ and $C_{\rm est}$ are the estimated matrices of  $G$ and $C$, respectively.
Here, we estimate the matrices $F_{\rm est}$ and $G_{\rm est}$ using \eqref{F} and \eqref{G}, respectively.
Using $T$ and $\hat{T}$, we define the following relative $H^2$ and $H^{\infty}$ norms:
\begin{align*}
g_2:= \frac{||T-T_{\rm est}||_{H^2}}{||T||_{H^2}}\quad {\rm and}\quad g_{\infty} := \frac{||T-T_{\rm est}||_{H^{\infty}}}{||T||_{H^{\infty}}}.
\end{align*}

Tables \ref{table_N200_2}, \ref{table_N400_2}, \ref{table_N600_2}, \ref{table_N800_2}, and \ref{table_N1000_2} present values of
$||e(\Theta_{\rm est})||_2^2$, $g_2$, $g_{\infty}$, and $\lambda_{\rm max}(F_{\rm est})$ for different $K$, as given by estimating $\Theta_{\rm est}$, $F_{\rm est}$, $G_{\rm est}$, and $C_{\rm est}$
using Algorithm 2, ${\rm CG}_1$, ${\rm CG}_2$, and a combined CG approach called Hybrid CG.
Hybrid CG is a combination of ${\rm CG}_1$ and ${\rm CG}_2$ obtained by applying
 ${\rm CG}_1$ for the first 15 iterations and ${\rm CG}_2$ for the next 5 iterations.
Note that in terms of the various indices,
we confirmed that Algorithm 2 provides a considerably better initial point $\Theta_0$ in Algorithm 1 than randomly choosing $\Theta_0\in M$.

According to Tables \ref{table_N200_2}, \ref{table_N400_2}, \ref{table_N600_2}, \ref{table_N800_2}, and \ref{table_N1000_2},
the results for
$||e(\Theta_{\rm est})||_2^2$, $g_2$, $g_{\infty}$, and $\lambda_{\rm max}(F_{\rm est})$ given by ${\rm CG}_1$ and ${\rm CG}_2$ are better than those given by Algorithm 2 for all $K$.
The results from ${\rm CG}_1$ and ${\rm CG}_2$ are almost the same for all $K$.
However, with the exception of $\lambda_{\rm max}(F_{\rm est})$,  the results from Hybrid CG are superior to those of ${\rm CG}_1$ and ${\rm CG}_2$.
Even when the combination of iterations was changed, the results of Hybrid CG were better than those of ${\rm CG}_1$ and ${\rm CG}_2$ in many cases.
Moreover, we should note that
the evaluation results may be worse as the data length $K$ increases.

\begin{table}[t]
\caption{Evaluation results when $K=200$ and ${\rm SNR}=15.043$.} \label{table_N200_2}
  \begin{center}
    \begin{tabular}{c|cccc} \hline
  $K=200$    &  $||e(\Theta_{\rm est})||^2_2$  & $g_2$  &   $g_{\infty}$  & $\lambda_{\rm max}(F_{\rm est})$    \\ \hline \hline \vspace{3mm}
Algorithm 2  & 23.097  & 0.660   & 0.868 &  $-0.138$ \\ \hline            
CG$_1$       & 4.675     &  0.288 & 0.328  & $-0.075$ \\ \hline  
CG$_2$       & 4.662     & 0.287  & 0.322   & $-0.075$\\ \hline
Hybrid CG   & 4.759     & 0.241      &  0.258  &  $-0.098$  \\ \hline
   \end{tabular}
  \end{center}
\caption{Evaluation results when $K=400$ and ${\rm SNR}=15.286$.} \label{table_N400_2}
  \begin{center}
    \begin{tabular}{c|cccc} \hline
  $K=400$    &  $||e(\Theta_{\rm est})||^2_2$  & $g_2$  &   $g_{\infty}$  & $\lambda_{\rm max}(F_{\rm est})$    \\ \hline \hline \vspace{3mm}
Algorithm 2  & 76.360  & 0.576    & 0.774 &  $-0.117$ \\ \hline            
CG$_1$       & 5.231      &  0.124 & 0.123  & $-0.064$ \\ \hline  
CG$_2$       & 5.226     & 0.124   & 0.122   & $-0.064$\\ \hline
Hybrid CG   & 4.900     & 0.111      &  0.087  &  $-0.065$  \\ \hline
   \end{tabular}
  \end{center}
\caption{Evaluation results when $K=600$ and ${\rm SNR}=15.669$.} \label{table_N600_2}
  \begin{center}
    \begin{tabular}{c|cccc} \hline
  $K=600$    &  $||e(\Theta_{\rm est})||^2_2$  & $g_2$  &   $g_{\infty}$  & $\lambda_{\rm max}(F_{\rm est})$    \\ \hline \hline \vspace{3mm}
Algorithm 2  & 75.850  & 0.624    & 0.835 &  $-0.258$ \\ \hline            
CG$_1$       & 7.028      &  0.144  & 0.218  & $-0.067$ \\ \hline  
CG$_2$       & 7.025     & 0.144  & 0.217   & $-0.067$\\ \hline
Hybrid CG   & 6.538     & 0.072      &  0.052  &  $-0.100$  \\ \hline
   \end{tabular}
  \end{center}
\caption{Evaluation results when $K=800$ and ${\rm SNR}=16.714$.} \label{table_N800_2}
  \begin{center}
    \begin{tabular}{c|cccc} \hline
  $K=800$    &  $||e(\Theta_{\rm est})||^2_2$  & $g_2$  &   $g_{\infty}$  & $\lambda_{\rm max}(F_{\rm est})$    \\ \hline \hline \vspace{3mm}
Algorithm 2  & 358.798  & 0.653  & 0.868 &  $-0.387$ \\ \hline            
CG$_1$       & 59.391      &  0.343  & 0.263  & $-0.115$ \\ \hline  
CG$_2$       & 59.000     & 0.343 & 0.265 & $-0.115$\\ \hline
Hybrid CG   & 20.623     & 0.191      &  0.211  &  $-0.129$  \\ \hline
   \end{tabular}
  \end{center}
\caption{Evaluation results when $K=1000$ and ${\rm SNR}=16.380$.} \label{table_N1000_2}
  \begin{center}
    \begin{tabular}{c|cccc} \hline
  $K=1000$    &  $||e(\Theta_{\rm est})||^2_2$  & $g_2$  &   $g_{\infty}$  & $\lambda_{\rm max}(F_{\rm est})$    \\ \hline \hline \vspace{3mm}
Algorithm 2  & 320.573  & 0.620   & 0.831 &  $-0.237$ \\ \hline            
CG$_1$       & 17.541      &  0.141  & 0.058  & $-0.082$ \\ \hline  
CG$_2$       & 17.464     & 0.141  & 0.057   & $-0.082$\\ \hline
Hybrid CG   & 12.865     & 0.094      &  0.055  &  $-0.093$  \\ \hline
   \end{tabular}
  \end{center}
\end{table}


\section{Conclusion and future work} \label{Sec7}
We developed identification methods for linear continuous-time symmetric systems using Riemannian optimization.
For this, we formulated three least-squares problems of minimizing the sum of squared errors on Riemannian manifolds,
and described the geometry of each problem.
In particular, we examined the quotient geometry in one problem in depth.
We proposed Riemannian CG methods for the three problems, and
 selected initial points using the modified MOESP method.
The results from a series of numerical simulations demonstrated the effectiveness of our proposed methods with comparisons to the traditional GN method.

The following problems should be addressed in future studies.
\begin{enumerate}
\item As mentioned in Remark \ref{remark_future}, system \eqref{1} does not correspond to a symmetric continuous-time system discussed in \cite{willems1976realization, tan2001stabilization}.
To identify such systems,
we need to develop a novel method that is fundamentally different from the methods proposed in this paper.

\item In Section \ref{Sec6B2}, we confirmed that the results produced by Hybrid CG,  a combination of ${\rm CG}_1$ and ${\rm CG}_2$, were better than those of ${\rm CG}_1$ and ${\rm CG}_2$ in many cases.
It would be interesting to study how the combination of iterations of ${\rm CG}_1$ and ${\rm CG}_2$ should be determined.

\item Lemma 2 in \cite{brockett1976some} shows that the manifold of transfer functions of SISO systems, i.e., $m=p=1$, is partitioned into multiple connected components. 
Thus, it is expected that $N/O(n)$ with $m=p=1$ will have multiple connected components, because each element in $N/O(n)$ corresponds to a transfer function.
If this is the case, different initial points on the different connected components will converge to different points, and thus initial points on $N/O(n)$ may
 considerably affect the system identification results.
In fact, we have confirmed that Algorithm 2 provides a better initial point than a random choice.
This provides a practical insight, and so it would be interesting to study whether or not the expectation is true.

\item In this paper, we proposed methods for identifying a target system as \eqref{1} with no noise.
As illustrated in Section \ref{Sec6}, our proposed methods are effective for identifying \eqref{1}, even if the output data was noisy.
However, if we were to consider the effect of noise on our methods, we may be able to derive better algorithms.
Thus, it is desirable to extend our proposed methods under the consideration of noise.

\item 
We developed an indirect method for identifying a continuous-time model \eqref{1}.
That is, we first identified the corresponding discrete-time model, and then this model was transformed into the required continuous-time model.
In contrast to our approach, the aim of a direct approach is to identify the original continuous-time model without using a discrete-time model.
However, as explained in \cite{garnier2008direct}, a direct approach has several drawbacks compared with an indirect approach.
Nevertheless, a direct approach is important for system identification, because 
it can provide physical insights into the continuous-time system to be identified.
Thus, it would be desirable to develop a direct approach for identifying a continuous-time model \eqref{1}.
\end{enumerate}


%


\section*{Acknowledgment}

This work was supported by JSPS KAKENHI Grant Numbers JP18K13773 and JP16K17647, and the Asahi Glass Foundation.

\appendix

\subsection{Geometry of the manifold ${\rm Sym}_+(n)$} \label{ape0}
We review the geometry of ${\rm Sym}_{+}(n)$ to develop optimization algorithms for solving our problems.
For a detailed explanation, see \cite{sato2018TAC, smith2005covariance}.

For $\xi_1$, $\xi_2\in T_S {\rm Sym}_{+}(n)$, we define Riemannian metric into ${\rm Sym}_+(n)$ as
\begin{align}
\langle \xi_1, \xi_2\rangle_S := {\rm tr} ( S^{-1} \xi_1 S^{-1} \xi_2  ). \label{metric}
\end{align}
Let $f: {\rm Sym}_+(n) \rightarrow {\bf R}$ be a smooth function and $\bar{f}$ be the extension of $f$ to Euclidean space ${\bf R}^{n\times n}$.
Riemannian gradient ${\rm grad}\, f(S)$ with respect to Riemannian metric \eqref{metric} is given by
\begin{align}
{\rm grad}\, f(S) = S {\rm sym}(\nabla \bar{f}(S)) S, \label{gradient}
\end{align}
where $\nabla \bar{f}(S)$ denotes the Euclidean gradient of $\bar{f}$ at $S\in {\rm Sym}_+(n)$.
The exponential map on ${\rm Sym}_+(n)$ is given by
\begin{align}
{\rm Exp}_S (\xi) &= S^{\frac{1}{2}} \exp (S^{-\frac{1}{2}} \xi S^{-\frac{1}{2}} ) S^{\frac{1}{2}} \nonumber\\
&= S\exp (S^{-1}\xi),\label{8}
\end{align}
where $\exp$ is the matrix exponential function.

We note that Riemannian metric \eqref{metric} is essentially the same with Fisher information metric
\begin{align*}
g_{S}^{\rm FIM} := {\rm E}({\rm D}l_x(S)\otimes {\rm D}l_x(S)), 
\end{align*}
where 
\begin{align*}
l_x(S):= \log p(x|S^{-1}),
\end{align*}
${\rm E}$ is the expectation operator with respect to $p(x|S^{-1})$,
and $\otimes$ is the tensor product.
Here, $p(x|S^{-1})$ denotes the Gaussian distribution with zero mean vector and  covariance $S^{-1}\in {\rm Sym}_+(n)$, that is,
\begin{align*}
p(x|S^{-1}) = \sqrt{\frac{{\rm det}\,S}{(2\pi)^n}} \exp \left( -\frac{1}{2} x^{\top} S x\right).
\end{align*}
Thus, $l_x(S)$ is the log-likelihood function of $p(x|S^{-1})$, and
\begin{align}
l_x(S) = -\frac{n}{2} \log (2\pi) + \frac{1}{2} \log \det S - \frac{1}{2}x^{\top}Sx. \label{loglike}
\end{align}
To see the relation between \eqref{metric} and $g_S^{\rm FIM}$, we use
\begin{align}
g_S^{\rm FIM} = -{\rm E}({\rm D}^2\,l_x(S)), \label{Fisher_metric}
\end{align}
where ${\rm D}^2\,l_x(S):{\rm Sym}(n)\times {\rm Sym}(n)\rightarrow {\bf R}$ is the second derivative of $l_x$ at $S$.
Eq.\,\eqref{Fisher_metric} can be found in Theorem 1 in \cite{smith2005covariance}.
The directional derivative of $l_x:{\rm Sym}(n)\rightarrow {\bf R}$ at $S\in {\rm Sym}_+(n)$ along $\xi\in T_S {\rm Sym}_+(n)\cong {\rm Sym}(n)$ is given by
\begin{align}
{\rm D}l_x(S)[\xi] = \frac{1}{2}{\rm tr}(S^{-1}\xi) -\frac{1}{2}{\rm tr}(xx^{\top}\xi), \label{42}
\end{align}
where the first term of the right-hand side is obtained by using Jacobi's formula ${\rm D}\det S [\xi] = {\rm tr}({\rm det}(S)S^{-1} \xi)$.
We define the inner product of ${\rm Sym}(n)$ as ${\rm tr}(\xi_1\xi_2)$ for any $\xi_1,\xi_2\in {\rm Sym}(n)$.
Then, from \eqref{42}, the gradient of $l_x$ at $S\in {\rm Sym}_+(n)$ is provided as
\begin{align*}
\nabla l_x(S) = \frac{1}{2}(S^{-1}-xx^{\top}).
\end{align*}
Moreover, according to \cite{absil2008optimization}, the Hessian of $l_x$ at $S$ is given by
\begin{align}
{\rm Hess}\,l_x(S)[\xi] = {\rm D} \nabla l_x (S)[\xi] = -\frac{1}{2}S^{-1}\xi S^{-1}, \label{43}
\end{align}
and
\begin{align}
{\rm D}^2l_x(S)[\xi_1,\xi_2] = {\rm tr}({\rm Hess}\,l_x(S)[\xi_1]\xi_2). \label{44}
\end{align}
Substituting \eqref{43} into \eqref{44}, we obtain that
\begin{align}
{\rm D}^2l_x(S)[\xi_1,\xi_2] = -\frac{1}{2} {\rm tr}(S^{-1}\xi_1S^{-1}\xi_2). \label{Hessian}
\end{align}
That is, ${\rm D}^2\,l_x(S)[\xi_1,\xi_2]$ is independent of $x$.
Hence, from \eqref{metric}, \eqref{Fisher_metric}, and \eqref{Hessian}, we obtain that
\begin{align*}
g_S^{\rm FIM}(\xi_1,\xi_2) = -{\rm D}^2l_x(S)[\xi_1,\xi_2]= \frac{1}{2} \langle \xi_1,\xi_2\rangle_S.
\end{align*}
Thus, Riemannian metric \eqref{metric} is essentially the same with Fisher information metric $g_S^{\rm FIM}$.

\subsection{Quotient manifold theorem} \label{ape_quotient}
This appendix explains how to use the following quotient manifold theorem as shown in Theorem 21.10 in \cite{lee2013intro} in our discussion of Section \ref{Sec3C}.

\begin{proposition}
Suppose that $\mathcal{G}$ is a Lie group acting smoothly, freely, and properly on a smooth manifold $\mathcal{M}$.
Then, the orbit space $\mathcal{M}/\mathcal{G}$ is a topological manifold of dimension equal to
${\rm dim}\, \mathcal{M} - {\rm dim}\, \mathcal{G}$, and has a unique smooth structure with the property that the quotient map $\pi:\mathcal{M}\rightarrow \mathcal{M}/\mathcal{G}$ is a smooth submersion.
\end{proposition}

\noindent
Here, the action $\cdot$ of Lie group $\mathcal{G}$ on a smooth manifold $\mathcal{M}$ is called 
\begin{itemize}
\item free if
$\{g\in \mathcal{G}\,|\, g\cdot p = p\} = \{e\}$
for each $p\in \mathcal{M}$, where $e$ is the identity of $\mathcal{G}$.
\item proper if the map $f:\mathcal{G}\times \mathcal{M} \rightarrow \mathcal{M}\times \mathcal{M}$ defined by $(g,p)\mapsto (g\cdot p,p)$ is a proper map.
That is, for every compact set $K\in \mathcal{M}\times \mathcal{M}$, the preimage $f^{-1}(K)\subset \mathcal{G}\times \mathcal{M}$ is compact.
\end{itemize}

We can apply the quotient manifold theorem in our case,
if \eqref{action} is a free and proper action.
This is because
the orthogonal group $O(n)$ is a Lie group, and \eqref{action} is a smooth action on the
smooth manifold $N$.
We thus confirm that action \eqref{action} is free and proper in Section \ref{Sec3C}.

\subsection{Proof of Theorem \ref{thm3}} \label{apeA}

In this appendix, we provide the proof of Theorem \ref{thm3} without deriving specific expressions of the vertical and horizontal spaces.
More concretely, we prove a more general theorem, and point out that Theorem \ref{thm3} is a corollary of the general theorem.

Let $\mathcal{M}$ be a Riemannian manifold with the Riemannian metric $\langle \cdot, \cdot \rangle$,
and let $\mathcal{G}$ be a group that smoothly acts on $\mathcal{M}$.
Here, we call $\phi_g:\mathcal{M}\rightarrow \mathcal{M}$ a smooth group action if $\phi_g$ is smooth and satisfies the following.
\begin{enumerate}
\item  For any $g_1,g_2\in \mathcal{G}$ and any $x\in\mathcal{M}$, $\phi_{g_1g_2}(x)=\phi_{g_1}(\phi_{g_2}(x))$ holds.
\item For the identity element $1_{\mathcal{G}}\in \mathcal{G}$ and any $x\in\mathcal{M}$, $\phi_{1_\mathcal{G}}(x) = x$ holds.
\end{enumerate}  
We write the derivative map of $\phi_g$ at $x\in \mathcal{M}$ as ${\rm D}\phi_g(x): T_x \mathcal{M}\rightarrow T_{\phi_g(x)}\mathcal{M}$.
By definition,
\begin{align*}
{\rm D} \phi_{1_\mathcal{G}}(x) &= {\rm D}(\phi_{g^{-1}}\circ \phi_g)(x)  \\
&= {\rm D}\phi_{g^{-1}}(\phi_g(x)) \circ {\rm D}\phi_g(x), 
\end{align*}
and thus
\begin{align}
{\rm D}\phi_{g^{-1}}(\phi_g(x)) = ({\rm D}\phi_g(x))^{-1}. \label{37}
\end{align}
Let $\mathcal{M}/\mathcal{G}$ be a quotient Riemannian manifold of $\mathcal{M}$ with the canonical projection $\pi:\mathcal{M}\rightarrow \mathcal{M}/\mathcal{G}$.
That is,
$\pi (x) = [x]$ for any $x\in \mathcal{M}$,
where $[x]:= \{ x_1\in \mathcal{M}\, |\, x = \phi_g(x_1)\,\, {\rm for\,\, some}\,\, g\in \mathcal{G}\}$.
Let $\mathcal{V}_x:= T_x\pi^{-1}([x])$ be the vertical space in $T_x \mathcal{M}$, and let $\mathcal{H}_x$ be the horizontal space that is the orthogonal complement of $\mathcal{V}_x$ with respect to the metric $\langle \cdot ,\cdot \rangle$.
Let $V$ be a vector space in $T_x\mathcal{M}$, and let
\begin{align*}
{\rm D}\phi_g(x)[V]:= \{ {\rm D}\phi_g(x)[\xi]\,|\, \xi\in V\}.
\end{align*}

\begin{lemma} \label{lem2}
For any $g\in \mathcal{G}$ and $x\in \mathcal{M}$,
\begin{align}
\mathcal{V}_{\phi_g(x)} = {\rm D}\phi_g(x)[\mathcal{V}_x]. \label{lemma_ver}
\end{align}
\end{lemma}
\begin{proof}
Let $\xi \in \mathcal{V}_{\phi_g(x)}=T_{\phi_g(x)} \pi^{-1}([\phi_g(x)])$.
Then, there exists a curve $\gamma$ such that $\gamma(0) = \phi_g(x)$ and $\dot{\gamma}(0)=\xi$.
Because $\mathcal{G}$ acts on $\mathcal{M}$, $\gamma_0(t):= \phi_{g^{-1}}(\gamma(t))$ is on $\pi^{-1}([x])$.
We have that
$\gamma_0(0) = \phi_{g^{-1}}(\phi_g(x)) = x$,
and
$\dot{\gamma}_0(t) = {\rm D}\phi_{g^{-1}}(\gamma(t))[\dot{\gamma}(t)]$.
Hence, it follows from \eqref{37} that
\begin{align*}
\dot{\gamma}_0(0) &= {\rm D}\phi_{g^{-1}}(\phi_g(x))[\xi] \\
 &= ({\rm D}\phi_g(x))^{-1}[\xi] \in T_x\pi^{-1}([x]) = \mathcal{V}_x,
\end{align*}
and thus
$\xi \in {\rm D}\phi_g(x) [\mathcal{V}_x]$.
That is,
\begin{align*}
\mathcal{V}_{\phi_g(x)} \subset {\rm D}\phi_g(x)[\mathcal{V}_x].
\end{align*}
Considering the dimension of both sides, we obtain \eqref{lemma_ver}. \qed
\end{proof}

Lemma \ref{lem2} implies the following theorem.

\begin{theorem} \label{thm5}
Suppose that the group action $\phi_g$ is an isometry in terms of Riemannian metric $\langle \cdot, \cdot \rangle$; i.e.,
for any $g\in \mathcal{G}$ and any $\xi_1,\xi_2\in T_x\mathcal{M}$,
\begin{align}
\langle {\rm D}\phi_g(x) [\xi_1], {\rm D}\phi_g(x) [\xi_2]\rangle_{\phi_g(x)} = \langle \xi_1,\xi_2\rangle_x. \label{invariant}
\end{align}
Then,
\begin{align}
\mathcal{H}_{\phi_g(x)} = {\rm D}\phi_g(x) [\mathcal{H}_x]. \label{39}
\end{align}
\end{theorem}
\begin{proof}
Taking the orthogonal complement of both sides of \eqref{lemma_ver}, we have that
\begin{align}
\mathcal{H}_{\phi_{g}(x)} = ({\rm D}\phi_g(x)[\mathcal{V}_x)])^{\perp}. \label{40}
\end{align}
Because \eqref{invariant} holds, we obtain that
$\langle {\rm D}\phi_g(x) [\xi_1], {\rm D}\phi_g(x) [\xi_2]\rangle_{\phi_g(x)}  = \langle \xi_1,\xi_2\rangle_x = 0$
for any $\xi_1\in \mathcal{V}_x$ and $\xi_2\in \mathcal{H}_x$.
This means that ${\rm D}\phi_g(x) [\xi_2]\in ({\rm D}\phi_g(x)[\mathcal{V}_x])^{\perp}$, which yields
${\rm D}\phi_g(x)[\mathcal{H}_x] \subset ({\rm D}\phi_g(x)[\mathcal{V}_x])^{\perp}$.
Considering the dimension of both sides, we have that
\begin{align}
{\rm D}\phi_g(x)[\mathcal{H}_x] = ({\rm D}\phi_g(x)[\mathcal{V}_x])^{\perp}. \label{41}
\end{align}
It follows from \eqref{40} and \eqref{41} that \eqref{39} holds. \qed
\end{proof}

Theorem \ref{thm5} yields the following corollary.
\begin{corollary} \label{cor1}
Suppose that the group action $\phi_g$ is an isometry in terms of Riemannian metric $\langle \cdot, \cdot \rangle$; i.e., \eqref{invariant} holds
for any $g\in \mathcal{G}$ and any $\xi_1,\xi_2\in T_x\mathcal{M}$.
Then,
\begin{align}
\bar{\xi}_{\phi_g(x)} = {\rm D}\phi_{g}(x)[\bar{\xi}_x], \label{399}
\end{align}
where
$\bar{\xi}_{x}$ and $\bar{\xi}_{\phi_g(x)}$ are the horizontal lifts of $\xi\in T_{[x]}(\mathcal{M}/\mathcal{G})$ at $x\in \mathcal{M}$ and $\phi_g(x)\in \mathcal{M}$, respectively.
\end{corollary}
\begin{proof}
Because $\pi\circ \phi_g = \pi$,
\begin{align}
{\rm D}(\pi\circ \phi_g)(x)[\bar{\xi}_x] = {\rm D}\pi(x)[\bar{\xi}_x] = \xi, \label{3991}
\end{align}
where the second equality follows from the definition of the horizontal lift.
Moreover, by the chain rule,
\begin{align}
{\rm D}(\pi\circ \phi_g)(x)[\bar{\xi}_x] = {\rm D}\pi(\phi_g(x))[{\rm D}\phi_g(x)[\bar{\xi}_x]]. \label{3992}
\end{align}
It follows from \eqref{3991} and \eqref{3992} that
${\rm D}\pi(\phi_g(x))[{\rm D}\phi_g(x)[\bar{\xi}_x]] = \xi$,
and Theorem \ref{thm5} yields
${\rm D}\phi_g(x)[\bar{\xi}_x]\in \mathcal{H}_{\phi_g(x)}$.
By the definition of the horizontal lift, we obtain \eqref{399}. \qed
\end{proof}

Theorem \ref{thm3} follows from Corollary \ref{cor1}.
This is because the group action $\phi_U(\Theta):=U\circ \Theta$ is an isometry, as shown in \eqref{touka}.

\subsection{Proof of Theorem \ref{Thm1}} \label{Proof_Thm1}

Because $T_{\Theta} N = \mathcal{V}_{\Theta} \oplus \mathcal{H}_{\Theta}$, $\eta$ can be uniquely decomposed into
\begin{align*}
\eta = \eta^v + \eta^h,\quad \eta^v\in \mathcal{V}_{\Theta},\,\, \eta^h\in \mathcal{H}_{\Theta}.
\end{align*}
Since $\eta^v\in \mathcal{V}_{\Theta}$, there exists $X\in {\rm Skew}(n)$ such that
\begin{align*}
\eta^v=(-XA+AX, -XB, CX).
\end{align*}
Thus, $\eta^h$ can be described as
\begin{align*}
\eta^h = (a+XA-AX,b+XB,c-CX).
\end{align*}
Because $\eta^h\in \mathcal{H}_{\Theta}$, we obtain that
\begin{align*}
&{\rm sk}(2(a+XA-AX)A^{-1}+B(b+XB)^{\top} \\
+& C^{\top}(c-CX))=0.
\end{align*}
It follows from this equation that \eqref{linear} holds, because $a^{\top}=a$ and $X^{\top}=-X$.

\subsection{Proof of Theorem \ref{Thm2}} \label{Proof_Thm2}

Using the Kronecker product and vec-operator, the operators $\mathcal{L}_0$ and $\mathcal{L}_1$ have the matrix representations
$K_0= A^{-1}\otimes A+A\otimes A^{-1} - 2I_{n^2}$ and $K_1 = I_n\otimes (BB^{\top}+C^{\top}C) + (BB^{\top}+C^{\top}C)\otimes I_n$, where $\otimes$ denotes the Kronecker product.
Both are symmetric, and $K_1$ is positive semidefinite \cite{antoulas2005approximation}.
Thus, $\mathcal{L}_1\geq 0$.
Note also that both summands of $K_1$ and thus of $\mathcal{L}_1$ are positive semidefinite, whence 
\begin{align}
\mathcal{L}_1(X)=0 \Rightarrow (BB^{\top}+C^{\top}C)X=0. \label{19}
\end{align}

If $\lambda_j, \lambda_k\in \lambda(A)$ with corresponding orthonormal eigenvectors $v_j, v_k$, then
\begin{align*}
\mathcal{L}_0 (v_jv_k^{\top}) = \mu_{jk} v_jv_k^{\top},
\end{align*}
where $\mu_{jk}:= \frac{(\lambda_j-\lambda_k)^2}{\lambda_j \lambda_k}$.
From the $n$ orthonormal eigenvectors $v_j$, $j=1,2,\ldots, n$, of the matrix $A$, we thus obtain $n^2$ orthonormal eigenvectors $v_jv_k^{\top}$, $j,k=1,2,\ldots, n$, of the linear matrix mapping.
Because $\mu_{jk}\geq 0$ for all $j,k$, it follows that $\mathcal{L}_0\geq 0$.
Together with $\mathcal{L}_1\geq 0$, this implies that 
\begin{align}
{\rm Ker}\, \mathcal{L}= {\rm Ker}\, \mathcal{L}_1 \cap {\rm Ker}\, \mathcal{L}_0. \label{20}
\end{align}
See Fact 8.7.3 in \cite{Bern09}.
Moreover,
the kernel of $\mathcal{L}_0$ is spanned by the matrices $v_jv_k^{\top}+v_kv_j^{\top}$ and $v_jv_k^{\top}-v_kv_j^{\top}$ with $\lambda_j=\lambda_k$, $j,k = 1,2,\ldots, n$.
That is,
\begin{align*}
{\rm Ker}\, \mathcal{L}_0 \cap {\rm Skew}(n) = {\rm span}\{ v_jv_k^{\top}-v_kv_j^{\top}\,|\, \lambda_j=\lambda_k \}.
\end{align*}

The matrix $A$ can be expressed as 
\begin{align*}
A=V{\rm diag}(\lambda_{n_1} I_{n_1}, \lambda_{n_2}I_{n_2},\ldots, \lambda_{n_l} I_{n_l}) V^{\top},
\end{align*}
where $n_1+n_2+\cdots +n_l=n$,  $\lambda_1=\cdots =\lambda_{n_1}<\lambda_{n_1+1}=\cdots =\lambda_{n_2}<\cdots <\lambda_{n_1+n_2+\cdots +n_{l-1}+1}=\cdots =\lambda_{n_l}$,
and after suitable ordering and partitioning, $V=\begin{pmatrix}
V_1 & \cdots & V_l
\end{pmatrix}=\begin{pmatrix}
v_1 & \cdots & v_n
\end{pmatrix}$
is orthogonal with ${\rm Im}\, V_j = {\rm Ker}(\lambda_{n_j}I_{n}-A)$.
We thus obtain that
\begin{align*}
 &\,\, {\rm Ker}\, \mathcal{L}_0 \cap {\rm Skew}(n) \\
 =& \{ V{\rm diag}(S_1,S_2,\ldots, S_l)V^{\top}\,|\, S_j\in {\rm Skew}(n_j)\}.
\end{align*}
To see this, note that the right hand side is the linear subspace of ${\rm Skew}(n)$, spanned by 
\begin{align*}
v_jv_k^{\top}-v_kv_j^{\top} = V(e_je_k^{\top}-e_ke_j^{\top})V^{\top},
\end{align*}
where $\lambda_j=\lambda_k$ and $e_j$ is the $j$-th unit vector in ${\bf R}^n$.
Thus, it follows from \eqref{19} and \eqref{20} that $U\in {\rm Ker}\,\mathcal{L} \cap {\rm Skew}(n)$ implies $U=V{\rm diag}(S_1,S_2,\ldots, S_l)V^{\top}$ with $(BB^{\top}+C^{\top}C)U=0$.
In particular, we have that
\begin{align*}
\begin{cases}
0 =B^{\top}UV_j =B^{\top}V_jS_j \\
0 = CUV_j = CV_jS_j
\end{cases}
\end{align*}
for $j=1,2,\ldots, l$, and thus
\begin{align*}
{\rm Ker}\,(\lambda_{n_j}I_{n}-A) \cap {\rm Ker}\,B^{\top}\cap {\rm Ker}\, C \supset {\rm Im}\,(V_jS_j).
\end{align*}
Therefore,
\begin{align}
\dim({\rm Ker}\,(\lambda_{n_j}I_{n}-A) \cap {\rm Ker}\,B^{\top}\cap {\rm Ker}\, C) \geq {\rm rank}\,S_j. \label{21}
\end{align}
Because each $S_j\in {\rm Skew}(n_j)$ necessarily has even rank,
assumption \eqref{assumption} and \eqref{21} yield that ${\rm rank}\,S_j=0$ for $j=1,2,\ldots, l$, whence $U=0$.
This implies that 
\begin{align}
{\rm Ker}\,\mathcal{L}\cap {\rm Skew}(n)=\{0\}, \label{22}
\end{align} 
or equivalently ${\rm Ker}\,\mathcal{L}\subset {\rm Sym}(n)$.
Eq.\,\eqref{22} implies that $\mathcal{L}:{\rm Skew}(n)\rightarrow {\rm Skew}(n)$ is an automorphism.

\ifCLASSOPTIONcaptionsoff
  \newpage
\fi



\bibliographystyle{IEEEtran}




\end{document}